\documentclass[10pt, reqno]{amsart}
\usepackage{amsmath, amsthm, amssymb}
\usepackage{times}
\usepackage{color}
\usepackage{hyperref}
\usepackage{comment}
\usepackage{enumerate}
\usepackage{setspace}
\usepackage{pgfplots}
\usepackage{graphicx}
\usepackage{xcolor}
\usepackage{booktabs}
\usepackage{float}
\numberwithin{equation}{section}
\newcommand{\R}{\mathbb{R}}
\newcommand{\cD}{\mathcal{D}}
\newcommand{\pa}{\partial}
\newcommand{\ve}{\epsilon}
\newcommand{\od}{\mathring{d}}
\newcommand{\T}{\mathbb{T}}
\newcommand{\rd}{\,{\rm d}}

\newcommand{\ck}{c_k^\epsilon}
\newcommand{\dx}{\rd x}
\newcommand{\ds}{\rd s}
\newcommand{\dt}{\rd t}
\newcommand{\mk}{\mu_k^\epsilon}
\newcommand{\dep}{d^\epsilon}
\renewcommand{\hat}{\widehat}
\renewcommand{\div}{\rm div}
\newtheorem{proposition}{Proposition}[section] % Numbered by section
\newtheorem{theorem}{Theorem}[section]
\newtheorem{lemma}{Lemma}[section] % Numbered by lemma
\newtheorem{definition}{Definition}[section]

\newtheorem{remark}[theorem]{Remark}

\begin{document}
  \author{Hengrong Du}
  \address{Department of Mathematics and Computer Science\\
    Fisk University \\
    Nashville, Tennessee, USA}
  \email{hdu@fisk.edu}
  \thanks{}

  \author{Fizay-Noah Lee}
  \address{Department of Mathematics\\
    Vanderbilt University\\
    Nashville, Tennessee\\
    USA}
  \email{fizzynoah@gmail.com}
  \thanks{}

  \author{Gieri Simonett}
  \address{Department of Mathematics\\
    Vanderbilt University\\
    Nashville, Tennessee\\
    USA}
  \email{gieri.simonett@vanderbilt.edu}
  \thanks{}
  \date{\today}
  \title{Global existence of weak solutions \\ to a nematic electrolyte model in
  2D}

  \noindent

  \begin{abstract}
    In this paper, we study a system of partial differential equations modeling the dynamics of nematic electrolytes on a two-dimensional torus $\mathbb{T}^2$. The model couples the Poisson--Nernst--Planck equations for the evolution of ion concentrations and electrostatic potential with the Ericksen-Leslie equations for the flow and orientation of nematic liquid crystals. We prove the global existence of weak solutions to this system. The proof relies on a Ginzburg-Landau approximation scheme. Key steps include deriving uniform energy estimates and establishing the strong convergence of the director field by utilizing a Pohozaev-type identity to handle the lack of compactness in the Ericksen stress tensor.
  \end{abstract}

  \allowdisplaybreaks

  \maketitle
  \section{Introduction}

Nematic liquid crystals consist of rod-like molecules that exhibit long-range 
orientational order, where the molecular orientations are correlated over large 
distances, leading to an ordered arrangement in a fluid-like medium. This order 
is distinct from positional order, as the molecules can still move freely, but 
their orientations exhibit significant alignment along a preferential direction, 
typically defined by what is known as the director \cite{de1993physics}. 
The unique structure of nematic fluids leads to anisotropic properties, meaning 
that the physical properties such as elasticity, viscosity, and dielectric 
constants are directionally dependent due to the alignment of the molecular axes. 
For instance, the Ericksen-Leslie model \cite{ericksen1961conservation,leslie1968some} 
serves as the foundational theory, consisting of fluid dynamics described by 
the Navier-Stokes equations \cite{temam2024navier} alongside an additional stress 
tensor that encapsulates the effects of molecular orientation. The coupling between 
fluid flow and the director field is significant as it influences both 
electrokinetic and hydrodynamic properties \cite{sonnet2012dissipative}. 
Nematic electrolytes are a class of materials that combine the properties of 
nematic liquid crystals with the ionic conduction behaviors characteristic of 
electrolytes. A key feature of nematic electrolytes is their response to an 
applied electric field, which results in nonlinear electro-osmosis (NEO) \cite{zaltz}. 
The NEO effect is characterized by the nonlinear dependence of induced fluid 
velocity on the electric field strength. This phenomenon occurs because the 
periodic reorientation of the director field due to the applied field leads to 
an effective charge separation and fluid movement. The interactions leading to 
this effect are vital for applications in microfluids and electrokinetic devices, 
where controlled fluid flow is required without reversing flow direction \cite{Calderer2016}. 
The behavior of ionic species in nematic electrolytes is governed by their 
interactions with the director field and the resulting electrohydrodynamic flows, 
which can generate complex vortex structures. Both point vortices and vortex 
filaments may emerge, reflecting flow phenomena that are not observed in 
isotropic fluids. Numerical illustrations of these effects can be found in \cite{bavnas2021numerical}.

The mathematical analysis of nematic electrolytes lies in the coupling between 
the Poisson--Nernst--Planck (PNP) framework, which governs the evolution of 
charge carriers and the electric potential \cite{choi}, and the Ericksen--Leslie equations, 
which describe the macroscopic flow and molecular orientation of the nematic liquid crystal. For the sake of simplicity, we neglect boundary conditions and consider the problem in a periodic domain $\T^{3}$.
  The system state is described by the concentrations of $N$ charged ion species 
  $c_{k}:\T^{3}\times[0,T]\to \R_{+}$ ($k=1, \dots, N$) with valences $z_{k}\in \R \setminus \{0\}$,
  electrostatic potential $\Phi:\T^{3}\times[0,T]\to \R$, fluid velocity
  $v=(v^1, v^2, v^3):\T^{3}\times[0,T]\to \R^{3}$ and pressure $p:\T^{3}\times[0,T]\to \R$, 
  and the director $d=(d^1, d^2, d^3):\T^{3}\times[0,T]\to \R^{3}$, representing the mean 
  molecular orientation. These quantities evolve according to:

  \begin{equation}
    \left\{
    \begin{aligned}
      \pa_{t}c_{k}+v\cdot \nabla c_{k}    & =\nabla\cdot(c_{k}\cD_{k}\nabla \mu_{k}),\\
      -\nabla\cdot (\varepsilon(d)\nabla \Phi)    & =\sum_{k=1}^{N}z_{k}c_{k}=:\rho,\\
      \pa_{t}v+(v\cdot \nabla v)+\nabla p & =- \nabla\cdot (\nabla d\odot \nabla d)+\nabla\cdot\sigma + \nabla\cdot \left( (\nabla \Phi\otimes \nabla \Phi)\varepsilon(d) \right), \\
      \nabla\cdot v                       & =0, \\
      \gamma_{1}\od+\gamma_{2}D(v)d     & = \Delta d+|\nabla d|^{2}d+\gamma_{2}(d^{\top}D(v) d) d \\
      & \qquad +\varepsilon_{a}(\nabla \Phi\otimes \nabla\Phi)d-\varepsilon_{a}(d^{\top}(\nabla \Phi\otimes \nabla \Phi)d)d,\\
      |d|&=1.
    \end{aligned}
    \right. \label{eqn:MainPDE}
  \end{equation}
  Here, the \emph{anisotropic diffusion matrices} $\cD_{k}$ reflect the mobility of the $k$-th species. In the isotropic case, where $\cD_{k}$ is proportional to the identity matrix, this corresponds to the classical Einstein relation for electron mobility in a gas. The \emph{electrochemical potentials} are defined by $\mu_{k}:=\ln c_{k}+z_{k}\Phi$. The \emph{dielectric permittivity matrix} $\varepsilon(d)$ is given by
  \begin{equation*}
    \varepsilon(d)=\varepsilon_{\|}d\otimes d+\varepsilon_{\perp}({\rm I}_{3\times 3}-d\otimes d)=\varepsilon_{\perp}{\rm I}_{3\times 3}
    +(\varepsilon_{\|}-\varepsilon_{\perp})d\otimes d,
  \end{equation*}
  where $\varepsilon_{\|}$ and $\varepsilon_{\perp}$ represent the parallel and perpendicular dielectric permittivity regarding $d$ ($\varepsilon_a := \varepsilon_{\|}-\varepsilon_{\perp}$). The Leslie stress tensor $\sigma$ takes the form
  \begin{equation*}
    \sigma=\alpha_{1}(D(v)d\cdot d)d\otimes d+\alpha_{2}\od\otimes d+\alpha_{3}
    d\otimes \od+\alpha_{4}D(v)+\alpha_{5}D(v)d\otimes d+\alpha_{6}d\otimes D
    (v)d, 
  \end{equation*}
  where $\od=\pa_{t}d+v\cdot \nabla d-\Omega(v)d$ is the co-rotational derivative of $d$, and $\alpha_{1}, \dots, \alpha_{6}$ are material constants. The symmetric and anti-symmetric parts of the velocity gradient are denoted by $D(v):=\frac{1}{2}(\nabla v+(\nabla v)^{\top})$ and $\Omega(v):=\frac{1}{2}(\nabla v-(\nabla v)^{\top})$, respectively. 
  System \eqref{eqn:MainPDE} constitutes a highly nonlinear system of partial differential equations. The Poisson--Nernst--Planck equations \eqref{eqn:MainPDE}$_{1-2}$ describe the evolution of charged carriers and the electric potential, featuring an anisotropic dielectric permittivity $\varepsilon(d)$ that depends on the director $d$. The forced Navier-Stokes equations \eqref{eqn:MainPDE}$_{3-4}$ govern the fluid velocity and pressure, driven by elastic stresses and external forces induced by the director field and the electric potential gradient $\nabla \Phi$. Finally, equation \eqref{eqn:MainPDE}$_{5}$ describes the evolution of the director, incorporating the influence of electric forces, subject to the unit-length constraint $|d|=1$ given by \eqref{eqn:MainPDE}$_{6}$.

In a similar fashion to \cite{de2023uniqueness, du2022weak,huang2014regularity, li2016uniqueness}, in this manuscript, we consider \eqref{eqn:MainPDE} in the two-dimensional setting assuming that the domain is $\T^2$.  More precisely, $v=(v^1, v^2)$ is a planar vector field (i.e., $v^3=0$, $\partial_{x_3}v=0$), and then both $D(v)$ and $\Omega(v)$ are $2\times 2$ tensor fields. Furthermore, we assume that $\partial_{x_3}c_k=\partial_{x_3}\Phi=0$ for all $k=1, \dots, N$, and $\partial_{x_3}d=0$.  We denote $\hat{d}=(d^1, d^2)$ for any given $d=(d^1, d^2, d^3)$, then the reduced Leslie stress tensor is given by 
\begin{align*}
  \hat{\sigma}&=\alpha_1(D(v)\hat{d}\cdot\hat{d})\hat{d}\otimes\hat{d}+\alpha_2\hat{\od}\otimes\hat{d}+\alpha_3\hat{d}\otimes\hat{\od}+\alpha_4D(v)+\alpha_5(D(v)\hat{d})\otimes\hat{d}+\alpha_6\hat{d}\otimes(D(v)\hat{d}).
\end{align*}
Moreover, $\hat{\varepsilon}(d)= \varepsilon_{\perp} {\rm I}_{2\times 2} + (\varepsilon_{\|}-\varepsilon_{\perp})\hat{d}\otimes\hat{d}$. 
Then the system \eqref{eqn:MainPDE} reads 
  \begin{equation}
    \left\{
    \begin{aligned}
      \pa_{t}c_{k}+v\cdot \nabla c_{k}    & =\nabla\cdot(c_{k}\cD_{k}\nabla \mu_{k}),\\
      -\nabla\cdot (\hat\varepsilon(d)\nabla \Phi)    & =\sum_{k=1}^{N}z_{k}c_{k}=:\rho, \\
      \pa_{t}v+(v\cdot \nabla v)+\nabla p & =- \nabla\cdot (\nabla d\odot \nabla d)+\nabla\cdot\hat\sigma + \nabla\cdot \left( (\nabla \Phi\otimes \nabla \Phi)\hat\varepsilon(d) \right), \\
      \nabla\cdot v                       & =0, \\
      \gamma_{1}\hat\od+\gamma_{2}D(v)\hat{d}     & = \Delta \hat{d}+|\nabla d|^{2}\hat{d}+\gamma_{2}(\hat{d}^{\top}D(v) \hat{d}) \hat{d} \\
      & \qquad +\varepsilon_{a}(\nabla \Phi\otimes \nabla\Phi)\hat{d}-\varepsilon_{a}(\hat{d}^{\top}(\nabla \Phi\otimes \nabla \Phi)\hat{d})\hat{d},\\
      \gamma_1 (\partial_t d^3+v\cdot \nabla d^3)&=\Delta d^3+|\nabla d|^2 d^3+\gamma_2 (\hat{d}^{\top}D(v)\hat{d}) d^3-\varepsilon_a(\hat{d}^{\top}(\nabla \Phi\otimes \nabla \Phi)\hat{d})d^3, 
      \\
      |d|&=1.
    \end{aligned}
    \right. \label{eqn:MainPDEreduced}
  \end{equation}
Without loss of generality, to simplify the presentation, we use the variables and corresponding tensors from \eqref{eqn:MainPDE} throughout the remainder of the manuscript rather than those in \eqref{eqn:MainPDEreduced}.
  \subsection{Weak solutions}
  For the Ericksen--Leslie system, it has been shown recently that finite-time singularities can develop even from smooth initial data \cite{huang2016finite,lai2022finite}. Since the nematic electrolyte model \eqref{eqn:MainPDE} contains the Ericksen--Leslie system as a subsystem, it is expected to suffer from the same issue. Consequently, classical solutions to our system may not exist globally in time. Therefore, we adopt the following weak formulation, which naturally accommodates the presence of singularities and allows for solutions that exist globally in time.

  Before stating the definition, we introduce the relevant function spaces. We denote by $L^2_{\rm div}(\T^2)$ the space of square-integrable, divergence-free vector fields on $\T^2$, and by $H^1_{\rm div}(\T^2)$ the corresponding Sobolev space, i.e., $H^1_{\rm div}(\T^2) = \{v \in H^1(\T^2;\R^2) : \nabla \cdot v = 0\}$. 
 \goodbreak

  \begin{definition}
    \label{def:weak} Given $0<T<\infty$, the functions
    \begin{align}
       & v\in L^{\infty}(0,T; L^{2}_{\rm div}(\T^{2}))\cap L^{2}(0,T; H^{1}_{\rm div}(\T^{2})), \nonumber \\
       & c_{k}\in L^{\infty}(0, T; L^{\infty}(\T^{2}))\cap L^{2}(0, T; H^{1}(\T^{2})), \nonumber \\
       & \Phi\in L^{\infty}(0,T; H^{1}(\T^{2}))\cap L^{\infty}(0,T; L^{\infty}(\T^{2})), \nonumber \\
       & d\in L^{\infty}(0,T; H^{1}(\T^{2})),\quad |d(x,t)|=1\;{\rm a.e. }(x,t)\in \T^{2}\times(0, T) \label{eqn:hard_constraint}
    \end{align}
    are a \emph{weak solution} of \eqref{eqn:MainPDE} if $c_k\ge 0$, and the functions satisfy the
    following integral identities:
    \begin{align*}
       & \int_{0}^{T}\int_{\T^2}c_{k}\pa_{t}\phi_{k}+c_{k}v\cdot \nabla \phi_{k}\dx\dt=\int_{0}^{T}\int_{\T^2}c_{k}\mathcal{D}_{k}\nabla \mu_{k}\cdot \nabla \phi_{k}\dx\dt,                \\
       & \int_{0}^{T}\int_{\T^2}(\varepsilon(d)\nabla \Phi)\cdot \nabla \psi\dx\dt =-\int_{0}^{T}\int_{\T^2}\sum_{k=1}^{N}z_{k}c_{k}\psi\dx\dt,                                                     \\
       & \int_{0}^{T}\int_{\T^2}v \cdot \pa_{t}z+(v\otimes v):\nabla z \dx\dt\\
       &=\int_{0}^{T}\int_{\T^2}[-\nabla d\odot \nabla d+\sigma +(\nabla\Phi\otimes \nabla \Phi )\varepsilon(d)]:\nabla z\dx\dt, \\
       & \int_{0}^{T}\int_{\T^2}\gamma_{1}[d \cdot \pa_{t}n + (v\otimes d):\nabla n +(\Omega(v)d) \cdot n]-\gamma_{2}(D(v)d )\cdot n \dx\dt\\
       & = \int_{0}^{T}\int_{\T^2}\nabla d:\nabla n-|\nabla d|^{2}d\cdot n-\gamma_{2}(d\cdot D(v) d)(d\cdot n)            \dx\dt                                                    \\
       & \qquad+\int_{0}^{T}\int_{\T^2}-\varepsilon_{a}((\nabla \Phi\otimes \nabla \Phi) d)\cdot n+\varepsilon_{a}(d\cdot (\nabla \Phi\otimes \nabla \Phi) d) (d\cdot n)\dx\dt
    \end{align*}
    for every $\phi_{k}, \psi\in C^{\infty}_{c}((0, T)\times \T^{2})$, $z\in C_c^{\infty}((0, T)\times \T^{2}; \R^{2})$, $n\in C_c^{\infty}((0, T)\times \T^{2}; \R^{3})$ with $\nabla \cdot z=0$ and the initial conditions:
    \begin{equation}
      c_{k}(x, 0)=c_{k, 0}(x)\ge c>0, \; v(x, 0)=v_{0}(x), \; d(x, 0)=d_{0}(x)
    \end{equation}
   for some constant $c>0$.
  \end{definition}

  \subsection{Assumptions}
  Throughout this paper, we adopt the following assumptions:
  \begin{enumerate}[({A}1)]
    \item \label{assump:species} The system involves $N \ge 2$ ionic species, where the $k$-th species has valence $z_k \in \R \setminus \{0\}$.
    \item The diffusion matrices $\cD_{k}$ are bounded and strictly positive definite (i.e., $(\cD_{k}\xi)\cdot \xi\ge \alpha|\xi|^{2}$ for some $\alpha>0$).
    \item The dielectric matrix $\varepsilon(d)$ is strictly positive definite with $\varepsilon_{\perp} >0$ and $\varepsilon_{a}\ge 0$, satisfying $\varepsilon(d)\xi\cdot \xi \ge \varepsilon_{\perp} |\xi|^{2}$.
    \item The material constants $\gamma_{1}$ and $\gamma_{2}$ are related to the Leslie coefficients through the compatibility conditions $\gamma_{1}=\alpha_{3}-\alpha_{2}$ and $\gamma_{2}=\alpha_{6}-\alpha_{5}$.
    \item \label{assump:viscosity} The viscosity coefficients satisfy $\gamma_1>0$, $\alpha_4>0$, $\alpha_1\ge 0$, and 
    $$4(\alpha_5+\alpha_6)\gamma_1-(\gamma_2+\alpha_2+\alpha_3)^2>0.$$
    \item \label{assump:electroneutrality} We observe from \eqref{eqn:MainPDE} that for the Poisson equation to be well-posed on the torus, we need for electroneutrality to hold i.e. $\int_{\mathbb{T}^2}\rho(t)\dx=0$ for all $t$. We note that if we choose initial conditions $c_{k,0}$ so that electroneutrality is satisfied, then this property is propagated in time as can be seen by integrating the Nernst-Planck equations $\eqref{eqn:MainPDE}_1$ in space. Throughout this paper, we shall assume that initial conditions are chosen so that electroneutrality is satisfied.
  \end{enumerate}
  \begin{remark}
    We remark that under the Parodi relation $\gamma_2=\alpha_6-\alpha_5=\alpha_2+\alpha_3$, the condition in (A\ref{assump:viscosity}) becomes $\gamma_1>0$, $\alpha_4>0$, $\alpha_1 \ge 0$, and $(\alpha_5+\alpha_6)\gamma_1-\gamma_2^2 >0 $.
  \end{remark}
The system \eqref{eqn:MainPDE} was introduced in \cite{Calderer2016}, derived from conservation laws and a variational structure; that work also included numerical studies of a simplified one-dimensional version. Due to the strong coupling between the fluid, charge carriers, and director field, the system presents significant mathematical challenges. To the best of our knowledge, the global existence of solutions (weak or strong) remains an open problem. It is worth mentioning that local-in-time well-posedness for the hyperbolic version (with $\ddot{d}$) in $\R^3$ was established in \cite{ma2023incompressible}. Recently, the existence of weak solutions for a relaxed system on $\T^3$ was proven in \cite{feireisl2020weak}. Furthermore, in \cite{gieri2026} we establish the existence and uniqueness of strong solutions with criteria for global existence, and characterization of equilibria to \eqref{eqn:MainPDE} in both 2D and 3D. In the present paper, we establish the global existence of weak solutions for the full system \eqref{eqn:MainPDE} on $\T^2$. Our proof utilizes a concentration-cancellation argument applied to a sequence of approximate solutions, significantly extending the framework of \cite{feireisl2020weak}. We now state our main result:
  \begin{theorem}\label{thm:MainTheorem}
    Let the assumptions (A\ref{assump:species})--(A\ref{assump:electroneutrality}) hold and the initial data satisfy
    \begin{equation*} c_{k, 0}\in
    L^{\infty}(\T^{2}) \text{ with } c_{k,0}\ge c> 0,\; v_{0}\in L^{2}_{\rm div}(\T^{2}),\; d_{0}\in H^{1}( \T^{2}; \mathbb{S}^{2})
    \end{equation*}
    for some constant $c>0$. Then the PDE system \eqref{eqn:MainPDE} admits a
    weak solution in the sense of Definition \ref{def:weak}.
  \end{theorem}

  This paper is organized as follows: In Section~2, we introduce a Ginzburg--Landau
  approximated system, for which the existence of weak solutions has been
  established in \cite{feireisl2020weak}. We then derive several energy estimates that remain
  independent of the approximation parameter $\epsilon \in (0,1)$. In Section~3,
  we address the convergence of the Ericksen stress tensor, a critical term, and
  demonstrate a concentration cancellation using a Pohozaev-type argument.
  Finally, in Section~4, we take the limit as $\epsilon \to 0$ to obtain a weak
  solution to the PDE system~\eqref{eqn:MainPDE}.

  \section{The Ginzburg--Landau approximation and Uniform estimates}
A primary challenge in analyzing the system \eqref{eqn:MainPDE} arises from the nonconvex constraint $|d|=1$ (see \eqref{eqn:hard_constraint}). To address this difficulty, we employ a Ginzburg--Landau approximation, considering the following regularized system (with approximation parameter $\epsilon\in(0,1)$):
  \begin{equation}
    \left\{
    \begin{aligned}
      \pa_{t}\ck+v^{\epsilon}\cdot \nabla \ck                                         & =\nabla\cdot(\ck \cD_{k}\nabla \mk), \quad k=1,...,N   \\
      -\nabla\cdot (\varepsilon(\dep)\nabla \Phi^{\epsilon})                                  & =\sum_{k=1}^{N}z_{k}\ck=:\rho^{\epsilon},  \\
      \pa_{t}v^{\epsilon}+(v^{\epsilon}\cdot \nabla v^{\epsilon})+\nabla p^{\epsilon} & =- \nabla\cdot (\nabla \dep\odot \nabla \dep)+\nabla\cdot\sigma^{\epsilon}+ \nabla\cdot \left( (\nabla \Phi^{\epsilon}\otimes \nabla \Phi^\ve)\varepsilon(\dep) \right), \\
      \nabla\cdot v^{\epsilon}                                                        & =0,  \\
      \gamma_{1}\od^{\epsilon}+\gamma_{2}D(v^{\epsilon})\dep                        & = \Delta \dep-f_{\epsilon}(\dep)+\varepsilon_{a}(\nabla \Phi^{\epsilon}\otimes \nabla\Phi^{\epsilon})\dep,
    \end{aligned}
    \right. \label{eqn:GinzburgLandauApp1}
  \end{equation}
  where $f_{\epsilon}(d)=\partial \mathcal{F}_{\epsilon}(d)/\partial d$, and $\mathcal{F}_{\epsilon}(d)$
   is a singular Ginzburg--Landau potential:
  \begin{equation*}
    \mathcal{F}_{\epsilon}(d) =
    \begin{cases}
      \frac{1}{2\epsilon^2}F(|d|^2), & \text{if }|d|<2,    \\
      +\infty,                    & \text{if }|d| \geq 2.
    \end{cases}
  \end{equation*}
  It follows that for $|d|<2$, 
  \begin{equation}\label{eq:fed}
      f_{\ve}(d)=\frac{1}{\ve^{2}}F'(|d|^{2})\,d.
  \end{equation}
  Here, $F:[0, 4)\to [0, +\infty)$ is a smooth function satisfying the following properties:
  \begin{enumerate}[(F1)]
    \item $F \in C^{\infty}([0, 4);\mathbb{R}_+)$ and convex.
    \item \label{prop:F_min} $F(r) \ge 0$ for all $r \in [0, 4)$, and $F(r) = 0$ if and only if $r=1$.
    \item \label{prop:F_sing} $\lim_{r \to 4^-} F(r) = +\infty$.
  \end{enumerate}
  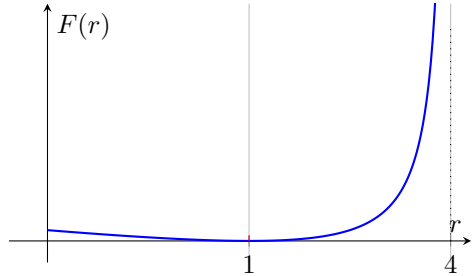
\begin{figure}[H]
  \centering
  \begin{tikzpicture}
    \begin{axis}[
      width=0.55\linewidth,   % adjust this
      height=5cm,             % and this
      axis lines=middle,
      xlabel={$r$},
      ylabel={$F(r)$},
      domain=0:1.95,
      samples=200,
      ymin=0,
      ymax=10,
      xtick={1,2},
      xticklabels={$1$, $4$},
      ytick=\empty,
      enlargelimits=true,
      grid=major
    ]
      \addplot[color=blue, thick] {(x - 1)^2 * (1 / (2 - x))};
      \draw[dashed, red] (axis cs:1,0) -- (axis cs:1,0.25);
      \draw[dotted, black] (axis cs:2,0) -- (axis cs:2,10);
    \end{axis}
  \end{tikzpicture}
  \vspace{-0.5em}  % optional: tightens the gap before the caption
  \caption{Plot of $F(r)$ with singularity at $r=4$.}
\end{figure}

  The system \eqref{eqn:GinzburgLandauApp1} differs from \eqref{eqn:MainPDE}
  primarily in the director equation, where the rigid constraint $|d|=1$ is
  relaxed via a Ginzburg--Landau penalization term $\mathcal{F}_{\epsilon}(d)$.
  Permitting the director field $d^{\epsilon}$ to deviate from the unit sphere
  improves its regularity from $H^1$ to $H^2$ \cite{feireisl2020weak}. A similar
  approach was utilized in \cite{lin1995nonparabolic} to establish the global
  existence of weak solutions for the simplified Ericksen--Leslie system
  (decoupled from the ionic concentrations $c$ and potential $\Phi$). The
  primary objective of this work is to analyze the singular limit as $\epsilon
  \to 0$ in \eqref{eqn:GinzburgLandauApp1}, recovering the rigid constraint
  $|d|=1$ via property (F\ref{prop:F_min}). While this limit has been
  established for the simplified Ericksen--Leslie system
  \cite{kortum2020concentration,du2022weak}, we extend the analysis here to the
  fully coupled nematic electrolyte model. In contrast, the fully coupled
  nematic electrolyte model includes additional nonlinear feedback between the
  director field, the ionic concentrations, and the electric potential. A key
  technical challenge is that the coupled director equation
  \eqref{eqn:GinzburgLandauApp1}$_5$ lacks a maximum principle. Consequently, the singular nature of the potential
  $\mathcal{F}_\ve$ (see F\ref{prop:F_sing}) is indispensable to strictly
  enforce the bound $|d^\epsilon| < 2$. For similar applications of singular
  potentials enforcing hard constraints, such as for the physical Q-tensor in
  the Landau--de Gennes model, we refer to
  \cite{feireislEvolutionNonisothermalLandau2014, wilkinson2015strictly,
  xu2022recent} and also in Cahn--Hilliard type equations 
  \cite{elliott1996cahn}.

  We consider a sequence of weak solutions $(c_1^\epsilon, \dots,
  c_N^\epsilon, \Phi^\epsilon, v^\epsilon, d^\epsilon)$ of \eqref{eqn:GinzburgLandauApp1}
  satisfying the initial conditions in Theorem \ref{thm:MainTheorem}, constructed as in
  \cite{feireisl2020weak}. These approximations enjoy uniform energy bounds
  together with the structural control provided by the singular potential
  $\mathcal{F}_\ve$. The singular potential is indispensable: it penalizes
  deviations from the unit-length constraint and enforces the strict bound
  $|d^\epsilon|<2$ in the approximate system. Our goal is to show that
  \begin{equation*}
  (c_1^\epsilon, \dots, c_N^\epsilon, \Phi^\epsilon, v^\epsilon, d^\epsilon)
  \end{equation*}
  converges to a solution $(c_1, \dots, c_N, \Phi, v, d)$ of \eqref{eqn:MainPDE} as $\epsilon \to 0$.

  \subsection{Energy estimate}
  We start with the energy estimate for \eqref{eqn:GinzburgLandauApp1}. Define the energy functional for the system \eqref{eqn:GinzburgLandauApp1} as
  \begin{equation*}
    E(t)=\int_{\T^2}\left( \frac{1}{2}|v^{\epsilon}|^{2}+\frac{1}{2}|\nabla \dep|
    ^{2}+\mathcal{F}_{\epsilon}(\dep)+\sum_{k=1}^{N}\ck\ln \ck+\frac{1}{2}\varepsilon(\dep)\nabla \Phi
    ^{\epsilon}\cdot \nabla \Phi^{\epsilon}\right)\dx. 
  \end{equation*}

  \begin{proposition}\label{prop:EnergyEst}
    \cite{feireisl2020weak, gieri2026}
    Let $(c^\ve_k, \Phi^\ve, v^\ve, d^\ve)$ be a sufficiently
    smooth solution to \eqref{eqn:GinzburgLandauApp1}.
    Then it holds that
    \begin{align}
         & E(t)+\sum_{k=1}^{N}\int_{0}^{t}\int_{\T^2}\alpha\ck|\nabla\mu^\ve_k|^{2}\dx\ds   \label{eqn:Energylaw}\\
         & +\int_{0}^{t}\int_{\T^2}(\alpha_{4}|D(v^{\epsilon})|^{2}+(\gamma_2+\alpha_2+\alpha_3)(\od^{\epsilon}\cdot D(v^{\epsilon}) \dep)+\alpha_1(\dep\cdot D(v^{\epsilon}) \dep)^2\nonumber\\
         &\qquad+(\alpha_{5}+\alpha_{6}) |D(v^{\epsilon}) \dep|^{2}+\gamma_1|\od^{\epsilon}|^{2})\dx\ds \nonumber\\
         & \le E(0).\nonumber
    \end{align}
  \end{proposition}
Throughout the remainder of the paper we write $E_{0}:=E(0)$ for the initial energy.
Furthermore, from the assumptions for the parameters $\alpha_1, \alpha_4, \alpha_5, \alpha_6$ (A\ref{assump:viscosity}), it follows from \eqref{eqn:Energylaw} and Korn's inequality that 
 $v^\ve$ is uniformly bounded in 
 \begin{equation*}
 L^\infty(0,T; L_{\rm div}^2(\T^2))\cap L^2(0,T;H_{\rm div}^1(\T^2)),
\end{equation*} 
i.e., 
\begin{equation*}
  \sup_{t \in [0,T]} \int_{\T^2} |v^\ve|^2 \dx + \int_0^T \int_{\T^2}|\nabla v^\ve|^2 \dx \rd t \le C E(0).
\end{equation*}
and
\begin{equation*}
  \sup_{t\in[0,T]}\int_{\T^2}|\nabla d^\ve(t)|^2 \dx+\int_0^T\int_{\T^2}(|\od^\ve|^2+|D(v^\ve) d^\ve|^2) \dx \rd t \le C E(0).
\end{equation*}

  \subsection{Uniform $L^\infty_tL^{2}_x\cap L^2_tH^1_x$ estimates for $\ck$. Uniform $L^\infty_tL^p_x$ estimates for $\nabla\Phi$}
  \label{ul2}

   We know that $\|\ck\|_{L^1}$ are
  uniformly bounded in time. In addition, from the previous energy estimate, we
  know that
  \begin{equation}
    \begin{split}
      \|\nabla \Phi^{\epsilon}(t)\|_{L^2}^{2}+ &\sum_{k=1}^{N}\int_{0}^{t}\int_{\T^2}
      \ck\varepsilon(d^\ve)\nabla\mk\cdot\nabla\mk \dx\ds \\
      &\lesssim \|\nabla \Phi^{\epsilon}(t)\|_{L^2}^{2}+ \sum_{k=1}^{N}\int_{0}^{t}\int_{\T^2}\ck|\nabla\mk |^{2}\dx\ds\le C
    \end{split}
  \end{equation}
  for some constant $C>0$ independent of $t$. Thus,
  \begin{align*}
    &\sum_{k=1}^{N}\int_{0}^{t}\int_{\T^2}\varepsilon_\perp\frac{|\nabla \ck|^{2}}{\ck}\dx \ds+ \sum_{k=1}^{N}\int_{0}^{t}\int_{\T^2}\varepsilon_\perp z_k^2\ck|\nabla\Phi^{\ve}|^{2}\dx\ds \\
    &\quad + 2\sum_{k=1}^{N}\int_{0}^{t}\int_{\T^2}z_{k}\nabla\ck\cdot \varepsilon(d^\ve)\nabla\Phi^{\epsilon}\dx\ds\le C
  \end{align*}
  from which it follows that
  \begin{align*}
        & \sum_{k=1}^{N}\int_{0}^{t}\int_{\T^2}\varepsilon_\perp\frac{|\nabla \ck|^{2}}{\ck}\,\dx \ds+ \sum_{k=1}^{N}\int_{0}^{t}\int_{\T^2}\varepsilon_\perp z_k^2\ck|\nabla\Phi^{\epsilon}|^{2}\dx\ds \\
        & \qquad +2\sum_{k=1}^{N}\int_{0}^{t}\int_{\T^2}\nabla(z_{k}\ck)\cdot(\varepsilon(d^\ve)\nabla\Phi^{\epsilon})\,\dx \ds\\
    =   & \sum_{k=1}^{N}\int_{0}^{t}\int_{\T^2}\varepsilon_\perp\frac{|\nabla \ck|^{2}}{\ck}\,\dx\ds + \sum_{k=1}^{N}\int_{0}^{t}\int_{\T^2}\varepsilon_\perp z_k^2 c_{k}^\ve|\nabla\Phi^{\epsilon}|^{2}\,\dx \ds
        +2\int_{0}^{t}\int_{\T^2}(\rho^{\epsilon})^{2}\,\dx\ds                                     \\
    =   & 4\sum_{k=1}^{N}\int_{0}^{t}\int_{\T^2}\varepsilon_\perp|\nabla \sqrt{\ck}|^{2}\,\dx\ds + \sum_{k=1}^{N}\int_{0}^{t}\int_{\T^2}\varepsilon_\perp z_k^2\ck|\nabla\Phi^{\epsilon}|^{2}\dx \ds
       +2\int_{0}^{t}\int_{\T^2}(\rho^{\epsilon})^{2}\,\dx \ds                                           \\
    \le & C.
  \end{align*}
  In particular, these calculations show that for each $k$,
  \begin{equation}
    \label{fisher}\sup_{t\in[0,T)}\int_{0}^{t}\int_{\T^2}|\nabla \sqrt{\ck}|^{2}\,\dx < C<
    \infty.
  \end{equation}
  We shall use this to show that $\|\ck\|_{L^2}$ are bounded uniformly on all
  finite time intervals. To do so, multiply the evolution equation for $\ck$ by
  $\ck$, sum in $k$, and integrate by parts to obtain
  \begin{align*}
    &\frac{1}{2}\frac{d}{dt}\sum_{k=1}^{N}\int_{\T^2}(\ck)^{2}\,\dx \\= & -\sum_{k=1}^{N}\int_{\T^2}\ck\mathcal{D}_{k}\nabla \mk\cdot \nabla \ck\,\dx\\
    = & -\sum_{k=1}^{N}\int_{\T^2}\mathcal{D}_{k}\nabla \ck\cdot \nabla \ck\,\dx -\sum_{k=1}^{N}\int_{\T^2}z_{k}\ck\mathcal{D}_{k}\nabla \Phi^{\epsilon}\cdot \nabla \ck\,\dx\\
    \lesssim & - \sum_{k=1}^{N}\|\nabla\ck\|_{L^2}^2
    + \sum_{k=1}^{N}\|\nabla \Phi^{\epsilon}\|_{L^2}^{\frac{1}{2}}\|\rho^{\epsilon}\|_{L^2}^{\frac{1}{2}}(\|\ck\|_{L^2}^{\frac{1}{2}}\|\nabla\ck\|_{L^2}^{\frac{1}{2}}+\|\ck\|_{L^2})\|\nabla \ck\|_{L^2} \\
    \lesssim & - \sum_{k=1}^N\|\nabla\ck\|_{L^2}^2+\|\rho^{\epsilon} \|_{L^2}^\frac{1}{2}\sum_{k=1}^N \|\ck\|_{L^2}^\frac{1}{2}\|\nabla \ck\|_{L^2}^\frac{3}{2}
    +\|\rho^\epsilon \|_{L^2}^\frac{1}{2} \sum_{k=1}^N \|\ck\|_{L^2}\|\nabla \ck\|_{L^2}\\
    \lesssim & - \sum_{k=1}^{N}\|\nabla\ck\|_{L^2}^2+\left(\sum_{k=1}^N \|\ck\|_{L^2}^\frac{1}{2}\right)^2\sum_{k=1}^N\|\nabla \ck\|_{L^2}^\frac{3}{2}
    + \left(\sum_{k=1}^N\|\ck\|_{L^2}\right)^\frac{3}{2} \sum_{k=1}^N\|\nabla \ck\|_{L^2}\\
    \lesssim & - \sum_{k=1}^{N}\|\nabla\ck\|_{L^2}^2 + \left(\sum_{k=1}^{N}\|\ck\|_{L^2}^{2}\right)^{2}+1
  \end{align*}
  where in the fourth line, we used the interpolation inequality $\|f\|_{L^4}\lesssim
  \|f\|_{L^2}^{\frac{1}{2}}\| f\|_{H^1}^{\frac{1}{2}}$ (for
  $f=\nabla\Phi^{\ve})$, and we also used the fact that $\|\Phi^{\ve}\|_{H^2}\lesssim
  \|\rho^{\ve}\|_{L^2}$ from elliptic regularity. In the fifth line, we used the
  fact that $\sup_{t}\|\nabla\Phi^{\ve}(t)\|_{L^2}$ is finite. In the sixth line, we used $\rho^\epsilon=\sum_kz_k\ck$ and the fact that for finite sums, for any $p\ge 0$, we have $\sum_k 
  |a_k|^p\lesssim \left(\sum_k |a_k|\right)^p\lesssim \sum_k|a_k|^p$. In the last line, we use this fact again, in addition to Young's inequality: for the second term, Young's inequality with exponents $4$ and $\frac{4}{3}$ respectively, and for the last term, exponents $\frac{8}{3}, 2$ and $8$, where it is understood that the exponent $8$ falls on the constant $1$, leading to the addition of a constant in the last line.

  Thus we arrive at the differential inequality
  \begin{align}\label{diffeq}
    \frac{d}{dt}\sum_{k=1}^{N}\|\ck\|_{L^2}^{2}\lesssim  - \sum_{k=1}^N \|\nabla \ck\|_{L^2}^2+\left(\sum_{k=1}^{N}\|\ck\|_{L^2}^{2}\right)^{2}+1
  \end{align}
  from which we conclude that
  \begin{align}
 \sum_{k=1}^{N}\|\ck(t)\|_{L^2}^{2}\le &\left(CT+\sum_{k=1}^{N}\|\ck(0)\|_{L^2}^{2}\right)    \label{l2est}
     \exp\left(C\int_{0}^{t}\sum_{k=1}^{N}\|\ck(s)\|_{L^2}^{2}\ds\right).
  \end{align}
      Indeed, given an absolutely continuous function $X:[0,T]\to\mathbb{R}^+$
  satisfying $\frac{dX}{dt}\le CX^{2}+C$, we have
  \begin{align*}
    &\frac{d}{dt}\left(X(t)\exp\left(-C\int_{0}^{t}X(s)\ds\right)\right) \\
    &\qquad= \left(\frac{dX}{dt}(t)-CX^{2}(t)\right)\exp\left(-C\int_{0}^{t}X(s)\ds\right)\le C.
  \end{align*}
  Integrating this inequality, we obtain
  $X(t)\le (CT+X(0))\exp\left(C\int_{0}^{t}X(s)\ds\right)$. Taking $X = \sum_{k=1}^{N}
  \|\ck\|_{L^2}^{2}$, this gives us \eqref{l2est}.

  On the other hand, from \eqref{fisher} we obtain
  \begin{align*}
    \int_{0}^{t}\|\ck(s)\|_{L^2}^{2}\ds & = \int_{0}^{t}\|\sqrt{\ck(s)}\|_{L^4}^{4}\ds                \\
     & \lesssim t+\int_{0}^{t}\|\nabla\sqrt{\ck(s)}\|_{L^2}^{2}\ds \\
     & \le C_{t}<\infty
  \end{align*}
  where in the second line, we used the interpolation inequality
  $\|f\|_{L^4}\lesssim \|f\|_{L^2}^{\frac{1}{2}}\|\nabla f\|_{L^2}^{\frac{1}{2}}+
  \|f\|_{L^2}$
  together with the fact that $\|\sqrt{\ck}\|_{L^2}^{2}=\|\ck\|_{L^1}$ is
  constant in time. From the preceding estimate and \eqref{l2est}, we conclude that
  \begin{equation}
    \label{cl2}\sup_{t\in[0,T)}\|\ck(t)\|_{L^2}<\infty
  \end{equation}
  for each $k$ and finite $T$. For later use, we state that, from the Poisson
  equation for $\Phi^{\epsilon}$, elliptic regularity, and the embeddings
  $H^{2}\hookrightarrow W^{1,p}$ for all finite $p$, an immediate consequence of
  \eqref{cl2} is that
  \begin{equation}
    \label{philp}\sup_{t\in[0,T)}\|\nabla\Phi^{\epsilon}(t)\|_{L^p}<\infty
  \end{equation}
  for all finite $p$ and $T$.
  We also note that as a consequence of \eqref{cl2} and \eqref{diffeq}, we have 
  \begin{align}\label{ch1}
      \sup_{t\in[0,T)}\int_0^t\|\nabla \ck(s)\|_{L^2}^2\,\ds<\infty.
  \end{align}

  \subsection{Uniform $L^{\infty}_tL^\infty_x$ estimates for $\ck$ and $\nabla\Phi$}
  In this subsection, we bootstrap from \eqref{cl2} to establish, for each $k$,
  \begin{equation}
    \label{alllp}\sup_{t\in[0,T)}\|\ck(t)\|_{L^p}<\infty
  \end{equation}
  for all finite $T$ and all $p\in[2,\infty]$.  We note, for future reference, that the above estimate, together with $\eqref{eqn:GinzburgLandauApp1}_2$ and elliptic regularity,  also implies
  \begin{align}   \label{phiin}   \sup_{t\in[0,T)}\|\nabla\Phi^\ve(t)\|_{L^\infty}<\infty.
  \end{align}
  Accordingly, we assume in this
  subsection that $T$ is fixed, and constants may depend on $T$.

  To begin, we multiply the evolution equation for $\ck$ by $(\ck)^{2m-1}$ for
  $m=2,3,4,...$ and integrate by parts, to obtain
  \begin{align*}
    &\frac{1}{2m}\frac{d}{dt}\int_{\T^2}(\ck)^{2m}\dx+\alpha(2m-1)\int_{\T^2}(\ck)^{2m-2}|\nabla \ck|^{2}\dx \\
    &\qquad\lesssim-(2m-1)\int_{\T^2} z_k(\ck)^{2m-1}(\mathcal{D}_k\nabla\Phi^{\epsilon}\cdot\nabla \ck)\dx.
\end{align*}
Hence,
    \begin{align*}\frac{1}{2m}\frac{d}{dt}\int_{\T^2}(\ck)^{2m}\dx+\alpha\frac{2m-1}{m^{2}}\|\nabla (\ck)^{m}\|_{L^2}^{2}\lesssim -\frac{2m-1}{m}\int_{\T^2}z_k(\ck)^{m}(\mathcal{D}_k\nabla\Phi^{\epsilon}\cdot\nabla (\ck)^{m})\dx.
  \end{align*}
  We point out that above, and for the remainder of this subsection, the constants
  implied by $\lesssim$ are independent of $m$.

  Multiplying the above estimate by $2m$, we obtain
  \begin{align}
    \label{cm}
    \frac{d}{dt}\|(\ck)^{m}\|_{L^2}^{2}&+\alpha\|\nabla (\ck)^{m}\|_{L^2}^{2} \nonumber \\
    &\le C(m)\left|\int_{\T^2}(\ck)^{m}(\mathcal{D}_k\nabla\Phi^{\epsilon}\cdot\nabla (\ck)^{m})\dx\right|,
  \end{align}
  where above, and in the remainder of the subsection, $C(m)$ refers to a
  constant that grows at most polynomially in $m$. Specifically, there exist
  $C, j>0$ such that $m^{j}\le C(m)\le Cm^{j}$ for $m=1,2,3,...$, but the $C$ and
  $j$ may differ from line to line.

  We estimate the integral on the right hand side, using \eqref{philp} and the interpolation
  inequality $\|f\|_{L^3}\lesssim \|f\|_{L^2}^{\frac{1}{2}}\|f\|_{H^1}^{\frac{1}{2}}$,
  \begin{align}
\label{int1}C(m)\left|\int_{\T^2}(\ck)^{m}\mathcal{D}_k\nabla\Phi^{\epsilon}\cdot\nabla(\ck)^{m}\dx\right|\le\nonumber & C(m)\|\nabla\Phi^{\epsilon}\|_{L^6}\|\nabla (\ck)^{m}\|_{L^2}\|(\ck)^{m}\|_{L^3} \\
    \le                          & \frac{\alpha}{2}\|\nabla(\ck)^{m}\|_{L^2}^{2}+C(m)\|(\ck)^{m}\|_{L^2}^{2}.
  \end{align}
  We further interpolate $\|(\ck)^{m}\|_{L^2}$ using the estimate
  $\|(\ck)^{m}\|_{L^2}^{2}\le \delta\|\nabla (\ck)^{m}\|_{L^2}^{2}+C\delta^{-1}\|
  (\ck)^{m}\|_{L^1}^{2}$, which, by taking $\delta$ small enough (depending on $m$),
  combined with \eqref{int1} yields
  \begin{align}
    \label{l1}C(m)\left|\int_{\T^2}(\ck)^{m}\nabla\Phi^{\epsilon}\cdot\nabla(\ck)^{m}\dx\right|\le \alpha\|\nabla(\ck)^{m}\|_{L^2}^{2}+C(m)\|(\ck)^{m}\|_{L^1}^{2}.
  \end{align}
  Thus, from \eqref{cm} and \eqref{l1}, we obtain

  \begin{align*}
    \frac{d}{dt}\|(\ck)^{m}\|_{L^2}^{2}\le C(m) \|(\ck)^{m}\|_{L^1}^{2},
  \end{align*}
  that is,
  \begin{align}
    \label{l2m}\frac{d}{dt}\|\ck\|_{L^{2m}}^{2m}\le C(m) \|\ck\|_{L^m}^{2m}.
  \end{align}
  Denoting
  $S_{m}(t) = \max\{\|\ck(0)\|_{L^\infty}, \sup_{s\in[0,t]}\|\ck(s)\|_{L^{m}}\}$
  and integrating \eqref{l2m} from $0$ to $t\in[0,T]$, we obtain
  \begin{align*}
    S_{2m}^{2m}(t) \le C(m)(1+t) S_{m}^{2m}(t)\le C(m)(1+T)S_{m}^{2m}(t),
  \end{align*}
  and so
  \begin{align*}
    S_{2m}(t)\le (C(m)(1+T))^{\frac{1}{2m}}S_{m}(t)\le C^{\frac{1}{2m}}m^{\frac{j}{2m}}S_{m}(t).
  \end{align*}
  Thus, considering $m = 2, 2^{2}, 2^{3},..., 2^{l}$, we have
  \begin{align*}
    S_{2^{l+1}}(t)\le C^{\frac{1}{2^{l+1}}}2^{\frac{jl}{2^{l+1}}}S_{2^l}(t),
  \end{align*}
  from which it follows that
  \begin{align}
    \label{S}S_{2^{l+1}}(t) \le C^{a_1}2^{a_2}S_{2}(t),
  \end{align}
  where $a_{1}= \sum_{l=1}^{\infty}\frac{1}{2^{l+1}}<\infty$ and
  $a_{2}= \sum_{l=1}^{\infty}\frac{jl}{2^{l+1}}<\infty$. Since the right hand side
  of \eqref{S} is independent of $l$, we pass to the limit $l\to\infty$ in \eqref{S}
  and find, from the definition of $S_{m}(t)$, that
  \begin{align}
    \label{lil2}\sup_{s\in[0,t]}\|\ck(s)\|_{L^\infty}\lesssim S_{2}(t).
  \end{align}
  The right hand side of \eqref{lil2} is finite, for all $t$, from subsection \ref{ul2};
  thus \eqref{alllp} is established.

\subsection{Positivity of $\ck$.} A similar iteration scheme as in the previous subsection, applied to the quantity $\frac{1}{\ck+\delta}$, and taking the limit $\delta\to 0^+$, yields positivity of $\ck$. We refer the reader to \cite{Lee23} for details.

\subsection{Uniform $L^{\infty}_t H^1_x$ estimates for $\ck$} Multiplying $\eqref{eqn:GinzburgLandauApp1}_1$ by $-\Delta \ck$ and integrating by parts, we obtain
\begin{align*}
    \frac{d}{dt}\|\nabla \ck\|_{L^2}^2+C\|\Delta \ck\|_{L^2}^2&\le C(\|v^\ve\|_{L^4}\|\nabla \ck\|_{L^4}+\|\nabla\Phi\|_{L^4}\|\nabla \ck\|_{L^4}+\|\Delta \Phi\|_{L^4}\|\ck\|_{L^4})\|\Delta \ck\|_{L^2}.
\end{align*}
Using the interpolation inequality $\|\nabla \ck\|_{L^4}\le C\|\nabla \ck\|_{L^2}^\frac{1}{2}\|\Delta \ck\|_{L^2}^\frac{1}{2}$ and the fact that $\ck$ is uniformly bounded in $L^\infty$ together with elliptic estimates applied to $\Phi$, the above inequality becomes
\begin{align*}
    \frac{d}{dt}\|\nabla \ck\|_{L^2}^2+C\|\Delta \ck\|_{L^2}^2&\le C((\|v^\ve\|_{L^4}+1)\|\nabla \ck\|_{L^2}^\frac{1}{2}\|\Delta \ck\|_{L^2}^\frac{1}{2}+1)\|\Delta \ck\|_{L^2}.
\end{align*}
Then, applying Young's inequality and absorbing into the left hand side, we obtain
\begin{align*}
    \frac{d}{dt}\|\nabla \ck\|_{L^2}^2+C\|\Delta \ck\|_{L^2}^2\le C+C(1+\|v^\ve\|_{L^4}^4))\|\nabla \ck\|_{L^2}^2.
\end{align*}
Then, applying Grönwall's inequality together with the fact that $v^\ve$ is uniformly bounded in $L^\infty_t L^2_x\cap L^2_tH^1_x\subset L^4_t L^4_x$, we conclude that 
\begin{align}\label{nac}
    \sup_{t\in[0,T)}\|\nabla \ck(t)\|_{L^2}<\infty.
\end{align}

  \section{Energy concentration and cancellation}
  \subsection{$\delta_{0}$-compactness}
  In this section, we state the $\delta_{0}$-compactness property for the
  family of director fields $\{d^{\ve}\}_{0<\ve<1}$ which yields the
  finiteness of the ``bad'' set on which the $H^{1}$ convergence fails.

  Consider the Ginzburg--Landau approximation with tensor field $\{\tau^{\ve}\}_{0<\ve<1}
  \subseteq L^{2}(\T^{2};\R^{3})$:
  \begin{equation}
    \Delta d^{\ve}-f_{\ve}(d^{\ve})=\tau^{\ve}\text{ in }\T^{2}. \label{eqn:GLtensor}
  \end{equation}
  We further assume that for $0<\Lambda_{1}, \Lambda_{2}<\infty$,
  \begin{equation}
    \sup_{0<\ve<1}\int_{\T^2}\left(\frac{1}{2}|\nabla d^{\ve}|^{2}+\mathcal{F}_{\ve}(d^{\ve}
    )\right)\rd x\le \Lambda_{1}, \label{eqn:Lambda1}
  \end{equation}
  and
  \begin{equation}
    \sup_{0<\ve<1}\|\tau^{\ve}\|_{L^2(\T^2)}\le \Lambda_{2}.\label{eqn:Lambda2}
  \end{equation}
  Then we can select a subsequence of $\{d^{\ve}\}_{0<\ve<1}$ (which we still denote by $\{d^\ve\}$ for simplicity) such that $d^{\ve}\rightharpoonup
  d$ in $H^{1}(\T^{2})$ and $d^{\ve}\rightarrow d$ in $L^{2}(\T^{2})$. The
  following lemma will play an essential role in our analysis:
  \begin{lemma}
    \label{lemma:compactness} There exists a constant $\delta_{0}=\delta_{0}(\Lambda
    _{1}, \Lambda_{2})>0$ such that if $\{d^{\ve}\}_{0<\ve<1}$ is a family of solutions
    to \eqref{eqn:GLtensor} satisfying \eqref{eqn:Lambda1}, \eqref{eqn:Lambda2},
    and if for some $x_{0}\in \T^{2}$ and $r_{0}>0$,
    \begin{equation}
      \sup_{0<\ve <1}\int_{B_{r_0}(x_0)}\left(\frac{1}{2}|\nabla d^{\ve}|^{2}+\mathcal{F}_{\ve}
      (d^{\ve})\right)\rd x\le \delta_{0}^{2},
    \end{equation}
    then there exists a map $d\in H^1(B_{{r_0}/4}(x_0);\mathbb{S}^2)$ such that after passing to a subsequence (still denoted by $d^\ve$), $d^{\ve}\to d$ in
    $H^{1}(B_{r_0/4}(x_{0}))$.
  \end{lemma}
 We remark that this lemma is a small-energy
  regularity result in the harmonic map and Ginzburg--Landau literature
  (see, e.g., \cite{chen1989existence,lin2008analysis}), and it is essential for
  the partial regularity theory, especially for controlling the size of the set of singularities. Our proof is
  similar in spirit to results for the simplified Ericksen--Leslie system
  \cite{du2022weak,kortum2020concentration,lin2016global}, but the singular
  potential makes the proof substantially more delicate. In particular,
  $\mathcal F_\ve(d)$ is not convex in $d$ because the minimizing set
  $\mathbb{S}^{2}$ is not convex, and the fact that $F$ blows up outside
  $|d|<2$ further complicates control of the potential term.
  \begin{proof}
    Fix $x_{1}\in B_{r_0/2}(x_{0})$. For $0<\ve<r_{0}/2$, we define $\hat{d}^{\ve}
    (x):=d^{\ve}(x_{1}+\ve x):B_{1}(0)\to \R^{3}$. Then we have
    \begin{equation}
      \label{eqn:hatdeq}\Delta \hat{d}^{\ve}=f_1(\hat{d}^{\ve})+\hat{\tau}^{\ve}\text{
      in }B_{1}(0),
    \end{equation}
    where $\hat{\tau}^{\ve}(x)=\ve^{2}\tau^{\ve}(x_{1}+\ve x)$. Let $\phi\in
    C_{0}^{\infty}(B_{1}(0))$ be a cut-off function such that $\phi=1$ on $B_{\frac{1}{2}}
    (0)$. Squaring both sides of \eqref{eqn:hatdeq}, multiplying by $\phi^2$, and integrating
    by parts, we obtain
    \begin{equation*}
      \begin{aligned}
        \int_{B_1(0)}|\hat{\tau}^{\ve}|^{2}\phi^{2}\dx & =\int_{B_1(0)}|\Delta\hat{d}^{\ve}-f_1(\hat{d}^{\ve})|^{2}\phi^{2}\dx \\
& =\int_{B_1(0)}\left(|\Delta \hat{d}^{\ve}|^{2}+|f_1(\hat{d}^{\ve})|^{2}\right)\phi^{2}+2\nabla \hat{d}^{\ve}:\nabla(f_1(\hat{d}^{\ve})\phi^{2})\dx  \\
& =\int_{B_1(0)}\left(|\Delta \hat{d}^{\ve}|^{2}+|f_1(\hat{d}^{\ve})|^{2}\right)\phi^{2} \\
& \quad +2\nabla\hat{d}^{\ve}:\nabla^{2}_{dd}\mathcal{F}_1(\hat{d}^{\ve})\nabla \hat{d}^{\ve}\phi^{2}+4\nabla\hat{d}^{\ve}:f_1(\hat{d}^{\ve})\otimes\phi\nabla\phi\dx \\
& \ge \int_{B_1(0)}\left(|\Delta \hat{d}^{\ve}|^{2}+|f_1(\hat{d}^{\ve})|^{2}\right)\phi^{2}-C_0|\nabla \hat{d}^\ve|^2 \phi^2 \\
& \quad -\frac{1}{2}|f_1(\hat{d}^{\ve})|^{2}\phi^{2}-8|\nabla\hat{d}^{\ve}|^{2}|\nabla\phi|^{2}\dx,
      \end{aligned}
    \end{equation*}
    where we utilize the structure of $\mathcal{F}_1$ and Young's inequality
    $4ab\ge -\frac{1}{2}a^{2}-8b^{2}$. More precisely, by the definition of the potential $\mathcal{F}_1(d) = \frac{1}{2}F(|d|^2)$, its Hessian is given by $$\nabla^2_{dd} \mathcal{F}_1(d) = F'(|d|^2) I + 2 F''(|d|^2) (d \otimes d).$$ For the double contraction with the gradient matrix $\nabla \hat{d}^{\epsilon}$, this yields
\begin{equation*}
  \nabla \hat{d}^{\epsilon} : \nabla^{2}_{dd}\mathcal{F}_1(\hat{d}^{\epsilon})\nabla \hat{d}^{\epsilon} = F'(|\hat{d}^{\epsilon}|^2) |\nabla \hat{d}^{\epsilon}|^2 + 2 F''(|\hat{d}^{\epsilon}|^2) | (\nabla \hat{d}^{\epsilon} )\hat{d}^{\epsilon}|^2.
\end{equation*}
By Property (F1), $F$ is convex, meaning $F'' \ge 0$, which ensures the second term is non-negative. Furthermore, the convexity of $F$ guarantees that $F'$ is monotonically non-decreasing. Thus, for any $r \in [0, 4)$, $F'(r) \ge F'(0)$. Setting $C_0 = -F'(0) > 0$, we obtain the structural lower bound
\begin{equation*}
  \nabla \hat{d}^{\epsilon} : \nabla^{2}_{dd}\mathcal{F}_1(\hat{d}^{\epsilon})\nabla \hat{d}^{\epsilon} \ge -C_0 |\nabla \hat{d}^{\epsilon}|^2.
\end{equation*}
     Thus $d^{\ve}\in H^{2}(B_{\frac{1}{2}}(0))$
    and $\|\widehat{d}^{\ve}\|_{H^{2}(B_{\frac{1}{2}}(0))}\leq C(\Lambda_1+\Lambda_{2})$.
    Hence by the Sobolev embedding theorem we have that
    $\widehat{d}^{\ve}\in C^{\frac{1}{2}}(B_{\frac{1}{2}}(0))$ and
    \begin{equation*}
      [\hat{d}^{\ve}]_{C^{\frac{1}{2}}(B_{\frac{1}{2}}(0))}\leq C\left\|\widehat{d}
      ^{\ve}\right\|_{H^{2}(B_{\frac{1}{2}}(0))}\leq C\left(\Lambda_1+\Lambda_{2}\right).
    \end{equation*}
    By rescaling, we get
    \begin{equation*}
      \left|d^{\ve}(x)-d^{\ve}(y)\right| \leq C\left(\Lambda_1+\Lambda_{2}\right)\left(\frac{|x-y|}{\ve}
      \right)^{\frac{1}{2}}, \quad \forall x, y \in B_{\ve/2}\left(x_{1}\right)
    \end{equation*}

    We claim that
    $\operatorname{dist}(d^{\ve}, \mathbb{S}^{2}) \leq \frac{1}{2}$ on
    $B_{\frac{r_{0}}{2}}\left(x_{0}\right)$. Suppose it were false. Then there
    exists $x_{1}\in B_{\frac{r_{0}}{2}}\left(x_{0}\right)$ such that
    $\operatorname{dist}\left(d^{\ve}\left(x_{1}\right), \mathbb{S}^{2}\right)>\frac{1}{2}$.
    Then for any $\theta_{0}\in(0,1)$ and $x \in B_{\theta_{0} \ve}\left(x_{1}\right
    )$, it holds
    \begin{equation*}
      \left|d^{\ve}(x)-d^{\ve}\left(x_{1}\right)\right| \leq C\left(\frac{\left|x-x_{1}\right|}{\ve}
      \right)^{\frac{1}{2}}\leq C \theta_{0}^{\frac{1}{2}}\leq \frac{1}{4}
    \end{equation*}
    provided $\theta_{0}\leq \frac{1}{16 C^{2}}$. It follows that
    \begin{equation*}
      \operatorname{dist}\left(d^{\ve}(x), \mathbb{S}^{2}\right) \geq \frac{1}{4}
      ,\quad \forall x \in B_{\theta_{0} \ve}\left(x_{1}\right),
    \end{equation*}
    so that
    \begin{equation*}
      \int_{B_{\theta_{0} \ve}\left(x_{1}\right)}\frac{1}{\ve^{2}}F(|d^\ve|^2)\rd x
      \geq C\pi \theta_{0}^{2}.
    \end{equation*}
    This contradicts the assumption that
    \begin{equation*}
      \int_{B_{\theta_{0} \ve}\left(x_{1}\right)}\frac{1}{\ve^{2}}F(|d^\ve|^2)\rd x
      \leq \int_{B_{r_{1}}(0)}\left(\frac{1}{2}\left|\nabla d^{\ve}\right|^{2}+\mathcal{F}_{\ve}
      (d^{\ve})\right)\dx \leq \delta_{0}^{2},
    \end{equation*}
    if we choose a sufficiently small $\delta_{0}>0$.

    The strategy now is to decompose $d^{\ve}$ into components tangential and normal
    to $\mathbb{S}^2$, derive a coupled system for the two components, and estimate
    each part separately using elliptic regularity.

    \textit{Step 1: Decomposition.}
    Since $\operatorname{dist}(d^{\ve},\mathbb{S}^{2})\le\frac{1}{2}$
    in $B_{r_0/2}(x_0)$, the field $d^{\ve}$ lies in the tubular neighborhood
    $\{d\in\mathbb{R}^3:\operatorname{dist}(d,\mathbb{S}^2)<1\}$, within which
    the nearest-point projection $d\mapsto d/|d|$ onto $\mathbb{S}^2$ is smooth.
    We therefore write
    \begin{equation*}
      d^{\ve}:=\rho_{\ve}\nu_{\ve},
    \end{equation*}
    where $\nu_{\ve}=\frac{d^{\ve}}{|d^{\ve}|}$ is the projection of $d^{\ve}$
    onto $\mathbb{S}^{2}$, $\rho_{\ve}=|d^\ve|$ is the magnitude of $d^\ve$.

    \textit{Step 2: Rewriting the equation.}
    Applying the Laplacian to $d^{\ve}=\nu_{\ve}\rho_{\ve}$ and
    expanding $$\Delta(\nu_{\ve}\rho_{\ve})=\rho_{\ve}\Delta\nu_{\ve}+2
    \nabla\rho_{\ve}\cdot\nabla\nu_{\ve}+\nu_{\ve}\Delta\rho_{\ve}$$ by the Leibniz rule,
    then substituting into \eqref{eqn:GLtensor} and expressing the nonlinear term
    $f_{\ve}(d^{\ve})$ in terms of $\rho_{\ve}$,
    we obtain
    \begin{equation}
     \rho_{\ve} \Delta \nu_{\ve}+\Delta \rho_{\ve}\nu_{\ve}+2 \nabla \rho_{\ve}\cdot \nabla \nu
      _{\ve} -\frac{1}{\ve^{2}}F'(\rho_{\ve}^{2}) \rho_\ve \nu_\ve=\tau^{\ve}.\label{eqn:decompeq}
    \end{equation}

    \textit{Step 3: Equation for the magnitude component $\rho_{\ve}$.}
    To extract an equation for $\Delta\rho_{\ve}$, we take the inner product of
    \eqref{eqn:decompeq} with $\nu_{\ve}$. We use the following identities, all
    following from $|\nu_{\ve}|=1$:
    \begin{itemize}
      \item Differentiating $|\nu_{\ve}|^2=1$ gives $( 2\nabla\rho_{\ve}\cdot
        \nabla\nu_{\ve})\cdot\nu_{\ve}=0$.
      \item Applying the Laplacian to $|\nu_{\ve}|^2=1$ gives
        $\nu_{\ve}\cdot\Delta\nu_{\ve}=-|\nabla\nu_{\ve}|^2$.
    \end{itemize}
    Using also the identity $\nu_{\ve}\rho_{\ve}\cdot\nu_{\ve}
    =\rho_{\ve}$, we obtain
    \begin{equation}
      \Delta \rho_{\ve}=
      \rho_{\ve}\left|\nabla \nu_{\ve}\right|^{2}+\frac{1}{\ve^{2}}F^{\prime}\left
      (\rho_{\ve}^{2}\right) \rho_{\ve}+\tau_{\ve}^{\perp},\label{eqn:rhove}
    \end{equation}
    where $\tau_{\ve}^{\perp}=\tau_{\ve}\cdot \nu_{\ve}$.

    \textit{Step 4: Equation for the orientational component $\nu_{\ve}$.}
    Substituting the equation \eqref{eqn:rhove} for $\Delta\rho_{\ve}$ back into \eqref{eqn:decompeq},
    the singular $\ve^{-2}$ terms cancel,  \begin{equation}\label{eqn:nuev}
     \rho_{\ve}\left(\Delta \nu_{\ve}+\left|\nabla \nu_{\ve}\right
      |^{2}\nu_{\ve}\right)+2\nabla \rho_{\ve}\cdot\nabla \nu_{\ve}=\tau_{\ve}^{\prime \prime},
    \end{equation}
    where $\tau_{\ve}^{\prime \prime}=\tau_{\ve}-\tau_{\ve}^{\perp}\nu_{\ve}$ is the
    component of $\tau_{\ve}$ tangential to $\nu_{\ve}$. Note that this equation
    contains no $\ve^{-2}$ factor, which is essential for obtaining $\ve$-uniform
    estimates on $\rho_{\ve}$.

    \textit{Step 5: $W^{2,\frac{4}{3}}$-estimates for  $\nu_{\ve}$.}
    Let $\eta \in C_{0}^{\infty}(B_{\frac{r_{0}}{2}}\left(x_{0}\right))$ be a
    standard cutoff function satisfying $\eta\equiv 1$ on
    $B_{\frac{3r_{0}}{8}}\left(x_{0}\right)$ and $0\le\eta\le 1$. We
    localize by multiplying by $\eta^2$ and apply the $W^{2,\frac{4}{3}}$-estimate
    separately to $\nu_{\ve}$.

    We compute $\Delta(\nu_{\ve}\eta^{2})$
    by multiplying the equation \eqref{eqn:nuev} for $\nu_{\ve}$ by $\eta^{2}$ and applying the
    Leibniz rule, obtaining
 \begin{equation}\label{eqn:Deltanuv-eta2}
  \begin{aligned}
    \Delta\!\left(\nu_{\ve}\eta^{2}\right)
    &= \nu_{\ve}\,\Delta(\eta^{2})+2\nabla(\eta^{2})\cdot \nabla \nu_\ve \\
    &\quad -\eta^{2}|\nabla\nu_{\ve}|^{2}\nu_{\ve}
       -\frac{2\eta^{2}}{\rho_{\ve}}\nabla\rho_{\ve}\cdot\nabla\nu_{\ve}
       +\frac{\eta^{2}}{\rho_{\ve}}\,\tau_{\ve}^{\prime\prime}.
  \end{aligned}
\end{equation}
  By the Calder\'on--Zygmund estimate $\|\nabla^{2}u\|_{L^{4/3}}\lesssim
\|\Delta u\|_{L^{4/3}}$ applied to $u=\nu_{\ve}\eta^{2}$ together with
\eqref{eqn:Deltanuv-eta2}, H\"older's inequality, and the bounds
$\rho_{\ve}\ge\tfrac12$, $|\nu_{\ve}|=1$, we obtain
\begin{equation*}%\label{eqn:nuv-W24/3}
  \begin{aligned}
    \bigl\|\nabla^{2}(\nu_{\ve}\eta^{2})\bigr\|_{L^{4/3}}
    &\lesssim \bigl\|\nu_{\ve}\Delta(\eta^{2})\bigr\|_{L^{4/3}}
       +\bigl\|\nabla\nu_{\ve}\cdot\nabla(\eta^{2})\bigr\|_{L^{4/3}} \\
    &\quad +\bigl\|\eta^{2}|\nabla\nu_{\ve}|^{2}\nu_{\ve}\bigr\|_{L^{4/3}}
       +\Bigl\|\tfrac{2\eta^{2}}{\rho_{\ve}}\nabla\rho_{\ve}\cdot\nabla\nu_{\ve}\Bigr\|_{L^{4/3}}
       +\Bigl\|\tfrac{\eta^{2}}{\rho_{\ve}}\tau_{\ve}^{\prime\prime}\Bigr\|_{L^{4/3}} \\
    &\lesssim 1+\|\nabla\nu_{\ve}\|_{L^{4/3}}
       +\|\nabla\nu_{\ve}\|_{L^{2}}\,\|\nabla(\nu_{\ve}\eta^{2})\|_{L^{4}} \\
    &\quad +\|\nabla(\nu_{\ve}\eta^2)\|_{L^{4}}\,\|\nabla\rho_{\ve}\|_{L^{2}}
       +\|\tau_{\ve}\|_{L^{2}} \\
    &\lesssim 1+\|\nabla d^{\ve}\|_{L^{2}}\Bigl[\,1
       +\|\nabla(\nu_{\ve}\eta^{2})\|_{L^{4}}
       \Bigr]
       +\|\tau_{\ve}\|_{L^{2}},
  \end{aligned}
\end{equation*}
where in the last line we used that $\nu_{\ve}=d^{\ve}/|d^{\ve}|$ and
$\rho_{\ve}=|d^{\ve}|$ are Lipschitz functions of $d^{\ve}$ on the
tubular neighborhood (see \cite[2.12.3]{Simon2012Theorems}), hence
$\|\nabla\nu_{\ve}\|_{L^{p}},\|\nabla\rho_{\ve}\|_{L^{p}}\lesssim
\|\nabla d^{\ve}\|_{L^{p}}$ for any $p\in[1,\infty]$. Now by the Sobolev inequality $\|\nabla(\nu_\ve \eta^2)\|_{L^4}\lesssim \|\nabla^2(\nu_\ve\eta^2)\|_{L^{4/3}}+1$ and the smallness assumption for $\|\nabla d^\ve\|_{L^2}$, we obtain the uniform estimate 
\begin{equation}\label{eqn:uniform_nuve_Est}
	\|\nabla \nu_\ve\|_{W^{2, 4/3}}\lesssim \|\tau_\ve\|_{L^2}+1.
\end{equation}
      \textit{Step 6: $H^1$ convergence of $\rho_\ve$.}
    We first show that $\rho_\ve\to 1$ in $L^2(B_{3r_0/8}(x_0))$.
    Throughout this step, all balls are centered at $x_0$.
    Recall that $\frac12\le\rho_\ve\le\frac32$ in $B_{r_0/2}$.
    For any $\delta\in(0,\frac12)$, Property (F2) and the continuity
    of $F$ imply that
    \begin{equation*}
      m_\delta:=
      \min_{\substack{s\in[1/2,3/2]\\|s-1|\ge\delta}}F(s^2)>0.
    \end{equation*}
    Hence, by the energy bound,
    \begin{equation*}
      \left|\left\{x\in B_{3r_0/8}:
      |\rho_\ve(x)-1|\ge\delta\right\}\right|
      \le \frac{1}{m_\delta}\int_{B_{3r_0/8}}
      F(\rho_\ve^2)\dx
      \le \frac{2\Lambda_1\ve^2}{m_\delta}.
    \end{equation*}
    Since $|\rho_\ve-1|\le\frac12$, it follows that
    \begin{equation*}
      \|\rho_\ve-1\|_{L^2(B_{3r_0/8})}^2
      \le \delta^2|B_{3r_0/8}|
      +\frac{\Lambda_1\ve^2}{2m_\delta}.
    \end{equation*}
    Sending first $\ve\to0$ and then $\delta\to0$, we obtain
    \begin{equation*}
      \rho_\ve\to1
      \quad\text{in }L^2(B_{3r_0/8}).
    \end{equation*}

    To prove the convergence of the gradients, choose
    $\eta\in C_0^\infty(B_{3r_0/8})$ such that
    $0\le\eta\le1$ and $\eta\equiv1$ on $B_{r_0/4}$.
    Multiplying \eqref{eqn:rhove} by
    $\eta^2(\rho_\ve-1)$ and integrating by parts, we obtain
    \begin{equation*}
      \begin{aligned}
        &\int_{B_{3r_0/8}}\eta^2|\nabla\rho_\ve|^2\dx
        +\frac{1}{\ve^2}\int_{B_{3r_0/8}}
        \eta^2F'(\rho_\ve^2)\rho_\ve(\rho_\ve-1)\dx\\
        &=-2\int_{B_{3r_0/8}}
        \eta(\rho_\ve-1)\nabla\eta\cdot\nabla\rho_\ve\dx\\
        &\quad-\int_{B_{3r_0/8}}
        \eta^2\rho_\ve(\rho_\ve-1)|\nabla\nu_\ve|^2\dx
        -\int_{B_{3r_0/8}}
        \eta^2\tau_\ve^\perp(\rho_\ve-1)\dx.
      \end{aligned}
    \end{equation*}
    By the convexity of $F$ and Property (F2),
    $F'(s)(s-1)\ge0$ for $s\in[0,4)$, and therefore
    \begin{equation*}
      F'(\rho_\ve^2)\rho_\ve(\rho_\ve-1)\ge0.
    \end{equation*}
    We may thus drop the potential term on the left-hand side.
    Applying H\"older's inequality and using
    $\rho_\ve\le\frac32$, we find
    \begin{equation*}
      \begin{aligned}
        \int_{B_{r_0/4}}|\nabla\rho_\ve|^2\dx
        &\lesssim
        \|\rho_\ve-1\|_{L^2(B_{3r_0/8})}
        \Big(
        \|\nabla\rho_\ve\|_{L^2(B_{3r_0/8})}\\
        &\qquad\qquad
        +\|\nabla\nu_\ve\|_{L^4(B_{3r_0/8})}^2
        +\|\tau_\ve\|_{L^2(B_{3r_0/8})}
        \Big).
      \end{aligned}
    \end{equation*}
    Here the implicit constant may depend on $r_0$, but is
    independent of $\ve$. The terms in parentheses are uniformly
    bounded: the first follows from
    $|\nabla\rho_\ve|\le|\nabla d^\ve|$ and
    \eqref{eqn:Lambda1}, the second from the local
    $W^{2,4/3}$-estimate for $\nu_\ve$ in Step 5 and the
    Sobolev embedding $W^{1,4/3}\hookrightarrow L^4$ in
    dimension two, and the third from \eqref{eqn:Lambda2}.
    Consequently,
    \begin{equation*}
      \int_{B_{r_0/4}}|\nabla\rho_\ve|^2\dx
      \lesssim\|\rho_\ve-1\|_{L^2(B_{3r_0/8})}
      \to0.
    \end{equation*}
    Together with the $L^2$ convergence established above,
    this yields
    \begin{equation*}
      \rho_\ve\to1
      \quad\text{strongly in }H^1(B_{r_0/4}).
    \end{equation*}

    \textit{Step 7: $H^1$ convergence of $d_\ve$.} It follows from the uniform
    $W^{2,4/3}$-estimate for $\nu_{\ve}$ \eqref{eqn:uniform_nuve_Est} and the
    Sobolev compact embedding $W^{2,4/3} \hookrightarrow H^1$ that there exists a $\nu\in H^1(B_{r_0/4} , \mathbb{S}^2)$ such that $\nu_\ve\to
    \nu$ in $H^1(B_{r_0/4})$. Since $\rho_{\ve}\to 1$ in $H^1(B_{r_0/4})$, we conclude that $d^{\ve}=\nu_{\ve}\rho_{\ve}\to \nu=:d$ in $H^1(B_{r_0/4})$. This completes the proof.
  \end{proof}

  Now we define the concentration set by
  \begin{equation*}
    \Sigma:=\bigcap_{r>0}\left\{x \in \T^{2}: \liminf_{\ve \rightarrow 0}\int_{B_{r}(x)}
    \left(\frac{1}{2}\left|\nabla d^{\ve}\right|^{2}+\mathcal{F}_{\ve}(d^{\ve})\right)\dx>\delta
    _{0}^{2}\right\},
  \end{equation*}
  where $\delta_{0}>0$ is given in Lemma \ref{lemma:compactness}.

  \begin{lemma}
    $\Sigma$ is a finite set and $|\Sigma|\le \Lambda_{1}/\delta_{0}^{2}$.
  \end{lemma}
  \begin{proof}
    We prove it by the following covering argument. Suppose
    $|\Sigma|\ge N>\Lambda_{1}/\delta_{0}^{2}$, then we can find $r>0$ and $\{x_{i}
    \}_{i=1}^{N}\subset \T^{2}$ such that $B_{r}(x_{i})\cap B_{r}(x_{j})=\emptyset$
    for $i\neq j$. Furthermore, we can find an $\ve\in(0, 1)$ such that for $i=1,
    2, \cdots, N$
    \begin{equation*}
      \int_{B_r(x_i)}\left(\frac{1}{2}|\nabla d^{\ve}|^{2}+\mathcal{F}_{\ve}(d^{\ve})\right
      )\rd x>\delta_{0}^{2}.
    \end{equation*}
    Then
    \begin{equation*}
      \begin{aligned}
        N\delta_{0}^{2} & <\sum_{i=1}^{N}\int_{B_r(x_i)}\left(\frac{1}{2}|{\nabla}d^{\ve}|^{2}+\mathcal{F}_{\ve}(d^{\ve})\right)\rd x \\
       & =\int_{\bigcup_iB_r(x_i)}\left(\frac{1}{2}|{\nabla}d^{\ve}|^{2}+\mathcal{F}_{\ve}(d^{\ve})\right)\rd x      \\
       & \le \int_{\T^2}\left(\frac{1}{2}|{\nabla}d^{\ve}|^{2}+\mathcal{F}_{\ve}(d^{\ve})\right)\rd x                \\
                        & \le \Lambda_{1}.
      \end{aligned}
    \end{equation*}
    This contradicts $N>\Lambda_{1}/\delta_{0}^{2}$.
  \end{proof}

  \subsection{Convergence on energy concentration sets}\label{3.2}

  In this section we aim to show the weak convergence of Ericksen stress tensors,
  i.e.,
  \begin{equation}\label{eqn:EricksenConv}
    \begin{split}
      \lim_{\epsilon\to 0}\int_{\T^2\times\left\{ t \right\}}&(\nabla d^{\ve}\odot \nabla
      d^{\ve}):\nabla z\dx =\int_{\T^2\times\left\{ t \right\}}(\nabla d\odot \nabla d)
      :\nabla z\dx.
    \end{split}
  \end{equation}
 Throughout this subsection we fix a time $t\in A$, where $A\subset[0,T]$ with $|A|=T$ denotes the set of times for which
  \begin{equation*}
    \liminf_{\ve\to0^+}\int_{\T^2}\big(|\nabla v^{\ve}|^{2}+|\od^{\ve}|^{2}+|D(v^{\ve})d^{\ve}|^{2}\big)(\cdot,t)\,\dx<\infty;
  \end{equation*}
  the existence of such a set is seen as follows. The energy estimate (Proposition \ref{prop:EnergyEst}) provides the uniform dissipation bound
  \begin{equation*}
    \sup_{0<\ve<1}\int_{0}^{T}\!\!\int_{\T^2}\big(|\nabla v^{\ve}|^{2}+|\od^{\ve}|^{2}+|D(v^{\ve})d^{\ve}|^{2}\big)\,\dx\rd t\le E_{0}.
  \end{equation*}
  Fix any sequence $\ve_{j}\to0^{+}$ and set
  \begin{equation*}
    g_{j}(t):=\int_{\T^2}\big(|\nabla v^{\ve_j}|^{2}+|\od^{\ve_j}|^{2}+|D(v^{\ve_j})d^{\ve_j}|^{2}\big)(\cdot,t)\,\dx\ \ge0 .
  \end{equation*}
  Since each $g_{j}$ is nonnegative and measurable on $[0,T]$, Fatou's lemma gives
  \begin{equation*}
    \int_{0}^{T}\liminf_{j\to\infty}g_{j}(t)\,\rd t\le \liminf_{j\to\infty}\int_{0}^{T}g_{j}(t)\,\rd t\le E_{0}.
  \end{equation*}
  Hence $t\mapsto\liminf_{j\to\infty}g_{j}(t)$ belongs to $L^{1}(0,T)$ and is, in particular, finite for a.e. $t\in[0,T]$. Since $\liminf_{\ve\to0^{+}}\le\liminf_{j\to\infty}$ along any fixed sequence $\ve_{j}\to0^{+}$, the set $A$ of times $t$ at which the displayed $\liminf$ is finite satisfies $|A|=T$. We first verify that $\{d^{\ve}(t)\}_{0<\ve<1}$ satisfies the hypotheses \eqref{eqn:Lambda1}--\eqref{eqn:Lambda2} of Lemma \ref{lemma:compactness}. To this end, observe that $d^{\ve}(t)$ solves \eqref{eqn:GLtensor} with
  \begin{equation*}
    \tau^{\ve}=\gamma_{1}\od^{\ve}+\gamma_{2}D(v^{\ve})d^{\ve}-\ve_{a}(\nabla\Phi^{\ve}\otimes\nabla \Phi^{\ve})d^{\ve},
  \end{equation*}
  and, along a subsequence realizing the $\liminf$ above, we estimate, using $|d^{\ve}|<2$ and the uniform estimate \eqref{philp} on $\nabla\Phi^{\ve}$,
  \begin{equation}\label{taubound}
    \|\tau^{\ve}(t)\|_{L^2(\T^2)}\lesssim \|\od^{\ve}(t)\|_{L^2}+\|D(v^{\ve})d^{\ve}(t)\|_{L^2}+\|\nabla\Phi^{\ve}(t)\|_{L^4}^{2}\le \Lambda_{2}<\infty,
  \end{equation}
  uniformly in $\ve$, while \eqref{eqn:Lambda1} holds with $\Lambda_{1}=E(0)$ by the energy estimate. Defining the concentration set at time $t$ by
  \begin{equation*}
    \Sigma_{t}:= \bigcap_{r>0}\left\{ \mathbf{x}\in \T^{2} : \liminf_{\ve\to0}\int
    _{B_r(\mathbf{x})\times\{t\}}\left( \frac{1}{2}|\nabla d^{\ve}|^{2}+ \frac{1}{\ve^{2}}
    F(|d^{\ve}|^2) \right) > \delta_{0}^{2}\right\},
  \end{equation*}
  with $\delta_{0}$ as in Lemma \ref{lemma:compactness}, the covering argument of the preceding subsection gives $|\Sigma_{t}|\le \Lambda_{1}/\delta_{0}^{2}$, and Lemma \ref{lemma:compactness} yields, after passing to a further subsequence, 
  \begin{equation*}
    d^{\ve}(t)\to d(t)\quad\text{in } H^{1}_{\rm loc}(\T^{2}\setminus\Sigma_{t}).
  \end{equation*}
 Here $d$ denotes the global limit obtained by the Aubin--Lions lemma in \eqref{eqn:vdconvergence}; after fixing a subsequence converging to $d(t)$ in $L^2(\T^2)$ for a.e.\ $t$ and restricting $A$ accordingly, the local limit is identified with $d(t)$.
  For simplicity, we assume the set of singularities
  $\Sigma_{t}=\{(0,0)\}\subset\T^{2}$ consists of a single point at zero. Let $z\in
  C^{\infty}(\T^{2}, \R^{2})$ be such that $\nabla\cdot z=0$ and $(0,0)\in \mbox{spt}
  (z)$. Then we observe that by adding the null term
  $-\frac{1}{2}|\nabla d^{\ve}|^{2}(\nabla \cdot z)$, we have
  \begin{equation*}
    \begin{aligned}
      &\int_{\T^2\times\left\{ t \right\}}(\nabla d^{\ve}\odot \nabla d^{\ve}):\nabla z \rd x\\
      & =\int_{\T^2\times\{t\}}(\nabla d^{\epsilon}\otimes \nabla d^{\epsilon}):\nabla z-\frac{1}{2}|\nabla d^{\epsilon}|^{2}(\nabla\cdot z)\rd x \\
      & =\int_{\T^2\times\left\{ t \right\}}\big(\nabla d^{\ve}\odot \nabla d^{\ve}-\frac{1}{2}|\nabla d^{\ve}|^{2}{\rm Id}\big):\nabla z \rd x.
    \end{aligned}
  \end{equation*}
  While by direct computations, we have
  \begin{equation*}
    \begin{split}
      \nabla d^{\ve}&\odot\nabla d^{\ve}-\frac{1}{2}|\nabla d^{\ve}|^{2}{\rm Id} \\
      &=\frac{1}{2}\left(
      \begin{array}{ll}
        |\partial_{x_1} d^\ve|^2-|\partial_{x_2} d^\ve|^2, & 
        2\partial_{x_1} d^\ve\cdot\partial_{x_2} d^\ve \\[1.5ex]
        2\partial_{x_1} d^\ve\cdot\partial_{x_2} d^\ve, & 
        |\partial_{x_2} d^\ve|^2-|\partial_{x_1} d^\ve|^2
      \end{array}
      \right).
    \end{split}
  \end{equation*}
  We can assume that there are two real numbers $\alpha, \beta$ such that
  \begin{equation}
    (|\partial_{x_1}d^{\ve}|^{2}-|\partial_{x_2}d^{\ve}|^{2})\dx\overset{\ast}{\rightharpoonup}(|\partial_{x_1}
    d|^{2}-|\partial_{x_2}d|^{2})\dx+\alpha \delta_{(0, 0)}, \label{convalpha}
  \end{equation}
  \begin{equation}
     (\partial_{x_1}d^{\ve}\cdot \partial_{x_2}d^{\ve}) \dx\overset{\ast}{\rightharpoonup}
     (\partial_{x_1}d\cdot\partial_{x_2}d)  \dx+\beta \delta_{(0, 0)}
    , \label{convbeta}
  \end{equation}
  hold as convergence of Radon measures. The existence of $\alpha,\beta\in\R$ is justified as follows: the measures on the left of \eqref{convalpha}--\eqref{convbeta} have total variation bounded by $\|\nabla d^{\ve}(t)\|_{L^2}^{2}\le 2\Lambda_{1}$, and hence, along a subsequence, converge weakly-$*$ to limit measures. By the strong convergence $d^{\ve}(t)\to d(t)$ in $H^{1}_{\rm loc}(\T^{2}\setminus\{(0,0)\})$, the corresponding defect measures, i.e., the differences between these limits and $(|\partial_{x_1}d|^{2}-|\partial_{x_2}d|^{2})\dx$, $(\partial_{x_1}d\cdot\partial_{x_2}d)\dx$, respectively, are supported in $\{(0,0)\}$, and a finite signed measure supported at a single point is a real multiple of the Dirac mass there. Next we want to show
  \begin{equation}
    \label{vanish}\alpha=\beta=0.
  \end{equation}
  Recall that
  \begin{equation}
    \begin{split}
      \Delta d^{\ve}&-f_{\ve}(d^{\ve})=\tau^{\ve}:=\gamma_{1}\od^{\ve}+\gamma_{2} D(v^{\ve})d^{\ve} -\ve_{a}(\nabla\Phi^{\ve}\otimes\nabla \Phi^{\ve})d^{\ve}.
    \end{split}
    \label{eqn:GinzburgLandauApp}
  \end{equation}
  For any $X\in C_{0}^{\infty}(\T^{2};\R^{2})$, we test the equation \eqref{eqn:GinzburgLandauApp}
  by $X\cdot \nabla d^{\ve}$ over $B_{r}(0)$. In what follows, $e_{\ve}(d^{\ve})$ denotes the Ginzburg--Landau energy density:
  \begin{equation*}
    e_{\ve}(d^{\ve}):=\frac{1}{2}|\nabla d^{\ve}|^{2}+\mathcal{F}_{\ve}(d^{\ve}).
  \end{equation*}
 We claim the pointwise identity
  \begin{equation*}
    \begin{split}
      (X\cdot\nabla d^{\ve})\cdot\big(\Delta d^{\ve}-f_{\ve}(d^{\ve})\big)
      =\ &\nabla\cdot\Big(\big(\nabla d^{\ve}\big)^\top\big(X\cdot \nabla d^{\ve}\big)-X\,e_{\ve}(d^{\ve})\Big) \\
      &+({\rm div}\,X)\,e_{\ve}(d^{\ve})-\nabla X:(\nabla d^{\ve}\odot \nabla d^{\ve}).
    \end{split}
  \end{equation*}
  To verify this, we expand the two divergences on the right-hand side:
  \begin{equation*}
    \begin{split}
      \nabla\cdot\big(\big(\nabla d^{\ve}\big)^\top(X\cdot \nabla d^{\ve})\big)
      &=(X\cdot\nabla d^{\ve})\cdot\Delta d^{\ve}+\nabla X:(\nabla d^{\ve}\odot \nabla d^{\ve})+X\cdot\nabla\big(\tfrac{1}{2}|\nabla d^{\ve}|^{2}\big),\\
      \nabla\cdot\big(X\,e_{\ve}(d^{\ve})\big)
      &=({\rm div}\,X)\,e_{\ve}(d^{\ve})+X\cdot\nabla\big(\tfrac{1}{2}|\nabla d^{\ve}|^{2}\big)+f_{\ve}(d^{\ve})\cdot(X\cdot\nabla d^{\ve}),
    \end{split}
  \end{equation*}
  where we used the chain rules $X\cdot\nabla\big(\mathcal{F}_{\ve}(d^{\ve})\big)=f_{\ve}(d^{\ve})\cdot(X\cdot\nabla d^{\ve})$ and $X\cdot\nabla\big(\tfrac{1}{2}|\nabla d^{\ve}|^{2}\big)=\nabla d^{\ve}:(X\cdot\nabla)\nabla d^{\ve}$, together with the symmetry of $\nabla d^{\ve}\odot\nabla d^{\ve}$. Subtracting the second line from the first, the terms $X\cdot\nabla\big(\tfrac{1}{2}|\nabla d^{\ve}|^{2}\big)$ cancel, and the claim follows. Integrating the pointwise identity over $B_{r}(0)$, using $\Delta d^{\ve}-f_{\ve}(d^{\ve})=\tau^{\ve}$, and applying the divergence theorem, with outward unit normal $\frac{x}{|x|}$ on $\partial B_{r}(0)$, we obtain the following Pohozaev identity
  \begin{equation}
    \label{eqn:pohozaev}
    \begin{aligned}
      & \int_{\pa B_r(0)}(X\cdot \nabla d^{\ve})\cdot\left(\frac{x}{|x|}\cdot \nabla d^{\ve}\right)\rd \sigma(x) \\
      & \quad -\int_{B_r(0)}\nabla X:\nabla d^{\ve}\odot \nabla d^{\ve}\rd{x} \\
      & \quad +\int_{B_r(0)}{\rm div}X e_{\ve}(d^{\ve})\rd{x} -\int_{\pa B_r(0)}e_{\ve}(d^{\ve})(X\cdot \frac{x}{|x|})\rd \sigma(x) \\
      & = \int_{B_r(0)}(X\cdot \nabla d^{\ve})\cdot \tau^{\ve}\rd{x}.
    \end{aligned}
  \end{equation}

  Let $X(x)=x$, then $\nabla X={\rm Id}$ and ${\rm div}\,X=2$, so that $\nabla X:(\nabla d^{\ve}\odot\nabla d^{\ve})=|\nabla d^{\ve}|^{2}$; since $2e_{\ve}(d^{\ve})-|\nabla d^{\ve}|^{2}=2\mathcal{F}_{\ve}(d^{\ve})=\frac{1}{\ve^{2}}F(|d^{\ve}|^{2})$, using the notation $\partial/\partial r$ to denote the directional derivative along $x/|x|$, i.e., $\partial d/\partial r=x/|x|\cdot \nabla d$, the identity \eqref{eqn:pohozaev} becomes
  \begin{equation*}
    \begin{split}
      &r\int_{\pa B_r(0)}\left|\frac{\pa d^{\ve}}{\pa r}\right|^2\rd\sigma(x)
      +\int_{B_r(0)}\frac{1}{\ve^{2}}F(|d^{\ve}|^2)\rd x -r\int_{\pa B_r(0)}e_{\ve} (d^{\ve})\rd\sigma(x)
      =\int_{B_r(0)}|x|\frac{\pa d^{\ve}}{\pa r}\cdot \tau^{\ve}\rd x.
    \end{split}
  \end{equation*}
  Hence,
  \begin{equation*}
    \begin{split}
      &\int_{\partial B_r(0)}e_{\ve}(d^{\ve})\rd\sigma (x) 
      = \int_{\partial B_r(0)}\left| \frac{\partial d^{\ve}}{\partial r}\right|^{2}\rd\sigma(x) + \frac{1}{r}\int_{B_r(0)}\frac{1}{\ve^{2}}F(|d^{\ve}|^2)\rd x + \mathcal{R}_{\ve}(r).
    \end{split}
  \end{equation*}
Here the remainder
  \begin{equation*}
    \mathcal{R}_{\ve}(r):=-\frac{1}{r}\int_{B_r(0)}|x|\,\frac{\pa d^{\ve}}{\pa r}\cdot \tau^{\ve}\rd x
    \quad\text{satisfies}\quad
    |\mathcal{R}_{\ve}(r)|\le \int_{B_r(0)}|\nabla d^{\ve}|\,|\tau^{\ve}|\rd x\le C_{0},
  \end{equation*}
  where we used $|x|\le r$ on $B_{r}(0)$, Cauchy--Schwarz, the energy estimate, and \eqref{taubound}; the constant $C_{0}$ is independent of $r$ and $\ve$.
  After integrating from $r$ to $R$, this yields
  \begin{equation}
    \label{eq:3.26}
    \begin{split}
      \int_{B_R(0)\setminus B_r(0)}e_{\ve}(d^{\ve}) \dx
      & = \int_{B_R(0)\setminus B_r(0)}\left| \frac{\partial d^{\ve}}{\partial r}\right|^{2}  \dx+ \int_{r}^{R}\frac{1}{\rho}\int_{B_\rho(0)}\frac{1}{\ve^{2}}F(|d^{\ve}|^2)\rd x\, \rd\rho \\
      & \quad + \int_{r}^{R}\mathcal{R}_{\ve}(\rho)\, \rd\rho.
    \end{split}
  \end{equation}
  Since $\Sigma_{t}=\{(0,0)\}$, we fix a radius $r_{*}>0$ and we may assume that there exists $\gamma\ge0$ such
  that
  \begin{equation*}
    e_{\ve}(d^{\ve})\,\dx\overset{\ast}\rightharpoonup \frac{1}{2}|\nabla
    d|^{2}\,\dx+ \gamma \delta_{(0,0)}\quad \text{in }B_{r_{*}}( 0),
  \end{equation*}
  as convergence of Radon measures; the existence of $\gamma\ge0$ follows as for $\alpha,\beta$ above, using in addition that $e_{\ve}(d^{\ve})\,\dx\ge0$, which forces the defect to be a \emph{nonnegative} multiple of $\delta_{(0,0)}$.
  We now pass to the limit $\ve\to0$ in \eqref{eq:3.26} term by term. Since $(0,0)\notin\overline{B_R(0)\setminus B_r(0)}$ and the limiting measure of $e_{\ve}(d^{\ve})\,\dx$ assigns no mass to $\partial B_r(0)\cup\partial B_R(0)$ (its singular part being concentrated at the origin), we have, for all $0<r<R<r_{*}$,
  \begin{equation*}
    \lim_{\ve\to0}\int_{B_R(0)\setminus B_r(0)}e_{\ve}(d^{\ve})\dx=\int_{B_R(0)\setminus B_r(0)}\frac{1}{2}|\nabla d|^{2}\dx.
  \end{equation*}
  For the first term on the right-hand side of \eqref{eq:3.26}, the strong convergence $d^{\ve}(t)\to d(t)$ in $H^{1}_{\rm loc}(\T^{2}\setminus\{(0,0)\})$ gives
  \begin{equation*}
    \lim_{\ve\to0}\int_{B_R(0)\setminus B_r(0)}\left|\frac{\partial d^{\ve}}{\partial r}\right|^{2}\dx=\int_{B_R(0)\setminus B_r(0)}\left|\frac{\partial d}{\partial r}\right|^{2}\dx;
  \end{equation*}
  for the second term, we apply Fatou's lemma in $\rho$; and for the third term, the uniform bound $|\mathcal{R}_{\ve}(\rho)|\le C_{0}$ established above gives
  \begin{equation*}
    \Big|\int_{r}^{R}\mathcal{R}_{\ve}(\rho)\,\rd\rho\Big|\le C_{0}(R-r)\le C_{0}R,
  \end{equation*}
  uniformly in $\ve$.
  Hence, after sending $\ve\to0$, from \eqref{eq:3.26} we obtain
\begin{align*}
  \int_{B_R(0)\setminus B_r(0)}\frac{1}{2}|\nabla d|^{2} \rd x
  &\ge \int_{B_R(0)\setminus B_r(0)}\left| \frac{\partial d}{\partial r}\right|^{2} \rd x\\
  &\quad + \int_{r}^{R}\frac{1}{\rho}\,\lim_{\ve\to0}\int_{B_\rho(0)}
      \frac{1}{\ve^{2}}F(|d^{\ve}|^2)\rd x\rd\rho - C_{0}R.
\end{align*}
  Sending $r\to0$, this further implies
  \begin{equation*}
    \int_{B_R(0)}\frac{1}{2}|\nabla d|^{2}\rd x\ge \int_{B_R(0)}\left| \frac{\partial
    d}{\partial r}\right|^{2}\rd x+ \int_{0}^{R}\frac{1}{\rho}\lim_{\ve\to0}\int_{B_\rho(0)}
    \frac{1}{\ve^{2}}F(|d^{\ve}|^2)\rd x \rd\rho - C_{0}R.
  \end{equation*}
  From this, we must have
  \begin{equation}
    \label{eq:3.28}\frac{1}{\ve^{2}}F(|d^{\ve}|^2) \to 0 \quad \text{in }L^{1}(B_{r_{*}}
    ).
  \end{equation}
  To see this, note that the nonnegative measures $\frac{1}{\ve^{2}}F(|d^{\ve}|^{2})\,\dx$ have total mass at most $2\Lambda_{1}$ and hence, along a subsequence, converge weakly-$*$ to a nonnegative measure $\mu$ on $B_{r_{*}}$. Away from the origin, $\mu$ vanishes: by the convexity of $F$ (Property (F1)) and $F\ge0=F(1)$, we have $F(r)\le F'(r)(r-1)$ for all $r$, so that on compact subsets of $B_{r_{*}}\setminus\{(0,0)\}$, where $\frac{1}{2}\le\rho_{\ve}=|d^{\ve}|<2$,
  \begin{equation*}
    \frac{1}{\ve^{2}}F(\rho_{\ve}^{2})\le \frac{1}{\ve^{2}}F'(\rho_{\ve}^{2})(\rho_{\ve}^{2}-1)\lesssim \frac{1}{\ve^{2}}F'(\rho_{\ve}^{2})\,\rho_{\ve}(\rho_{\ve}-1)\to 0\quad\text{in } L^{1}_{\rm loc}\big(B_{r_{*}}\setminus\{(0,0)\}\big),
  \end{equation*}
  where the convergence of the right-hand side was established in Step 6 of the proof of Lemma \ref{lemma:compactness} (both nonnegative terms on the left-hand side of the identity there tend to $0$). Therefore $\mu=\kappa\delta_{(0,0)}$ for some $\kappa\ge0$, and \eqref{eq:3.28} amounts to showing that $\kappa=0$.
  Otherwise,
  \begin{equation*}
    \frac{1}{\ve^{2}}F(|d^{\ve}|^2)\, \dx\overset{\ast}\rightharpoonup \kappa \delta_{(0,0)}
  \end{equation*}
  for some $\kappa>0$, which would imply
  \begin{equation*}
    \int_{0}^{R}\frac{1}{\rho}\lim_{\ve\to0}\int_{B_\rho(0)}\frac{1}{\ve^{2}}F(|d^{\ve}|^2
    )\,\rd x \rd\rho = \int_{0}^{R}\frac{\kappa}{\rho}\, d\rho = \infty,
  \end{equation*}
  a contradiction.

  Next, by choosing $X(x)=(x_1,0)$ in \eqref{eqn:pohozaev}, for which ${\rm div}\,X=1$, $\nabla X:(\nabla d^{\ve}\odot\nabla d^{\ve})=|\partial_{x_1}d^{\ve}|^{2}$, and
  $e_{\ve}(d^{\ve})-|\partial_{x_1}d^{\ve}|^{2}=\frac{1}{2}\big(|\partial_{x_2}d^{\ve}|^{2}-|\partial_{x_1}d^{\ve}|^{2}\big)+\mathcal{F}_{\ve}(d^{\ve})$,
  we obtain
\begin{equation}
  \label{eq:3.29}
  \begin{split}
    &\frac12\int_{B_r(0)}\bigl(|\partial_{x_2}d^{\ve}|^{2}
      -|\partial_{x_1}d^{\ve}|^{2}\bigr)\dx
      +\int_{B_r(0)}\mathcal{F}_{\ve}(d^{\ve})\dx \\
    &= \int_{B_r(0)} x_1\partial_{x_1}d^{\ve}\cdot\tau^{\ve}\dx
      -\int_{\partial B_r(0)} x_1\partial_{x_1}d^{\ve}\cdot\partial_r d^{\ve}\,\rd\sigma(x)  +\int_{\partial B_r(0)}\frac{x_1^{2}}{r} e_{\ve}(d^{\ve})\,\rd\sigma(x).
  \end{split}
\end{equation}
  Observe that by Fubini's theorem, for a.e.\ $r>0$,
  \begin{equation*}
    \int_{\partial B_r(0)}x_1\partial_{x_1}d^{\ve}\cdot \partial_{r}d^{\ve}\rd\sigma(x)
    \to \int_{\partial B_r(0)}x_1\partial_{x_1}d\cdot\partial_{r}d\rd\sigma(x),
  \end{equation*}
  \begin{equation*}
    \int_{\partial B_r(0)}\frac{x_1^{2}}{r}e_{\ve}(d^{\ve})\rd\sigma(x) \to \frac{1}{2}\int_{\partial
    B_r(0)}\frac{x_1^{2}}{r}|\nabla d |^{2}\rd\sigma(x),
  \end{equation*}
  and by \eqref{eq:3.28},
  \begin{equation*}
    \int_{B_r(0)}\frac{1}{\ve^{2}}F(|d^{\ve}|^2)\dx \to 0 .
  \end{equation*}
  Furthermore,
  \begin{equation*}
    \left| \int_{B_r(0)}x_1\partial_{x_1}d^{\ve}\cdot\tau^{\ve}\dx\right| \le
    C r \|\tau^{\ve}\|_{L^2}\|\nabla d^{\ve}\|_{L^2}\le C_{1}r,
  \end{equation*}
with $C_{1}$ independent of $r$ and $\ve$, by \eqref{taubound} and the energy estimate. Moreover, since the limiting measure in \eqref{convalpha} assigns no mass to $\partial B_r(0)$ for any $r>0$ (its singular part being concentrated at the origin), \eqref{convalpha} yields
  \begin{equation*}
    \lim_{\ve\to0}\int_{B_r(0)}\big(|\partial_{x_1}d^{\ve}|^{2}-|\partial_{x_2}d^{\ve}|^{2}\big)\dx=\int_{B_r(0)}\big(|\partial_{x_1}d|^{2}-|\partial_{x_2}d|^{2}\big)\dx+\alpha.
  \end{equation*}
  Hence, sending $\varepsilon\to0$ in \eqref{eq:3.29}, we obtain, for a.e. $r\in(0,r_{*})$,
  \begin{equation*}
    \Big|\int_{B_r(0)}\big( |\partial_{x_1}d|^{2}- |\partial_{x_{{2}}}d|^{2}\big) \dx+ \alpha \Big|\le C_{2}\,r+3r\int_{\partial B_r(0)}|\nabla d|^{2}\,\rd\sigma(x).
  \end{equation*}
  Here $C_{2}:=2C_{1}$, and the two boundary integrals in \eqref{eq:3.29} were estimated using the elementary bounds $|x_{1}|\le r$ and $\frac{x_{1}^{2}}{r}\le r$ on $\partial B_{r}(0)$.
  This implies $\alpha=0$ after sending $r\to0$ along a suitable sequence: the integral over $B_r(0)$ vanishes as $r\to0$ by the dominated convergence theorem, since $|\nabla d|^{2}\in L^{1}(\T^{2})$; and since
  \begin{equation*}
    \int_{0}^{r_{*}}\Big(\int_{\partial B_r(0)}|\nabla d|^{2}\,\rd\sigma(x)\Big)\rd r=\int_{B_{r_{*}}(0)}|\nabla d|^{2}\rd x<\infty,
  \end{equation*}
  we must have $\liminf_{r\to0^+}\,r\int_{\partial B_r(0)}|\nabla d|^{2}\,\rd\sigma(x)=0$, so there exists a sequence $r_{j}\to0^{+}$ along which the right-hand side tends to $0$. As $\alpha$ is independent of $r$, we conclude $\alpha=0$. Similarly, choosing $X(x)=(x_2,0)$, for which ${\rm div}\,X=0$ and $\nabla X:(\nabla d^{\ve}\odot\nabla d^{\ve})=\partial_{x_1}d^{\ve}\cdot\partial_{x_2}d^{\ve}$, in \eqref{eqn:pohozaev} and passing to the limit as above, now using \eqref{convbeta}, we obtain
  \begin{equation*}
    \Big|\int_{B_r(0)}(\partial_{x_1}d\cdot\partial_{x_2}d)\rd x + \beta \Big|\le C_{2}\,r+3r\int_{\partial B_r(0)}|\nabla d|^{2}\,\rd\sigma(x),
  \end{equation*}
  which implies $\beta=0$ after sending $r\to0$ along a suitable sequence, as before. This proves \eqref{convalpha},
  \eqref{convbeta}, and \eqref{vanish}. Hence, as $t\in A$ was arbitrary and the argument applies near each point of the finite set $\Sigma_{t}$, \eqref{eqn:EricksenConv} holds for every $t\in A$.

  \section{Existence of weak solutions}
  In this section, we complete the proof of existence for the system \eqref{eqn:MainPDE} 
  by passing to the limit as $\epsilon \to 0$. We begin by establishing uniform bounds 
  on the time derivatives to ensure compactness. We claim that there exists some 
  $p \in (1,2)$ such that
  \begin{equation*}
    \sup_{\ve>0}\left[ \|v_{t}^{\ve}\|_{L_t^2 W_{x,\div}^{-2,p}}+ \|d_{t}^{\ve}\|
    _{L_t^{4/3} L_x^{4/3}}\right] < \infty .
  \end{equation*}
  Here $W_{x,\div}^{-2,p}$ denotes the dual of
  \begin{equation*}
   W_{\div}^{2,p'}(\T^2) := \{ g \in W^{2,p'}(\T^2,\mathbb{R}^{2}) : \nabla
    \cdot g = 0 \}, \quad p' = \frac{p}{p-1}.
  \end{equation*}
  To see this, first observe that for $1<p<2$,
  \begin{equation*}
    - v^{\ve}\cdot \nabla v^{\ve}- \nabla \cdot (\nabla d^{\ve}\odot
    \nabla d^{\ve})+\nabla\cdot\sigma^\ve+\nabla\cdot((\nabla\Phi^\ve\otimes\nabla\Phi^\ve)\varepsilon(d^\ve)) \in L_{t}^{2}W_{x}^{-2,p},
  \end{equation*}
  and
  \begin{align*}
   & \big\|   - v^{\ve}\cdot \nabla v^{\ve}- \nabla \cdot (\nabla d^{\ve}\odot
    \nabla d^{\ve})+\nabla\cdot\sigma^\ve+\nabla\cdot((\nabla\Phi^\ve\otimes\nabla\Phi^\ve)\varepsilon(d^\ve))\big\|_{L_t^2 W_x^{-2,p}} \\
    & \le C \Big( \|v^{\ve}\otimes v^{\ve}\|_{L_t^2 L_x^1}+ \|\nabla d^{\ve}\odot \nabla d^{\ve}\|_{L_t^2 L_x^1}+\|\sigma^{\ve}\|_{L^2_tL^2_x}+\|(\nabla\Phi^{\ve}\otimes\nabla\Phi^{\ve})\varepsilon(d^{\ve})\|_{L^\infty_tL^2_x}\Big) \\
   & \le C \Big( (1+\|v^{\ve}\|_{L_t^\infty L_x^2}) \|v^{\ve}\|_{L_t^\infty L_x^2}+ \|\nabla d^{\ve}\|_{L_t^\infty L_x^2}^{2}\Big) \le C .
  \end{align*}
  The duality mechanism behind the estimate is the following: if $\nabla g\in L^{\infty}(\T^2)$, then for any tensor field $M\in L^{1}(\T^2)$,
  \begin{equation*}
    |\langle \nabla\cdot M, g\rangle|=\Big|\int_{\T^2}M:\nabla g\,\dx\Big|\le \|M\|_{L^1_x}\,\|\nabla g\|_{L^\infty_x},
  \end{equation*}
  and $\|\nabla g\|_{L^\infty_x}\lesssim\|g\|_{W^{2,p'}_x}$ by the two-dimensional Sobolev embedding $W^{1,p'}(\T^2)\hookrightarrow L^{\infty}(\T^2)$ (applied to $\nabla g$), which holds if and only if $p'=\frac{p}{p-1}>2$, that is, $p<2$; thus $\|\nabla\cdot M\|_{W^{-2,p}_x}\lesssim \|M\|_{L^1_x}$. In particular, for every $p\in(1,2)$,  the following estimates hold
  \begin{align*}
    \|\sigma^{\ve}\|_{L^2_tL^2_x}&\lesssim \big(1+\|d^{\ve}\|_{L^\infty(Q_T)}^{4}\big)\big(\|\od^{\ve}\|_{L^2_tL^2_x}+\|D(v^{\ve})\|_{L^2_tL^2_x}\big)\le C,\\
    \|(\nabla\Phi^{\ve}\otimes\nabla\Phi^{\ve})\varepsilon(d^{\ve})\|_{L^\infty_tL^2_x}&\lesssim \big(1+\|d^{\ve}\|_{L^\infty(Q_T)}^{2}\big)\|\nabla\Phi^{\ve}\|_{L^\infty_tL^4_x}^{2}\le C,
  \end{align*}
  by the energy estimate, the uniform bound $|d^{\ve}|<2$, and \eqref{philp}.

  Then we can estimate
  \begin{align*}
   & \|v_{t}^{\ve}(\cdot,t)\|_{W_{x,\div}^{-2,p}}\\
    & = \sup \Big\{ \langle v_{t}^{\ve}(\cdot,t), g \rangle : \|g\|_{W_{\div}^{2,p'}}\le 1 \Big\}\\
  & = \sup \Big\{ \big\langle - v^{\ve}\cdot \nabla v^{\ve} - \nabla \cdot (\nabla d^{\ve}\odot \nabla d^{\ve})  + \nabla\cdot\sigma^\ve+\nabla\cdot((\nabla\Phi^\ve\otimes\nabla\Phi^\ve)\varepsilon(d^\ve)), g \big\rangle \\
    & \qquad \quad : g \in W_{\div}^{2,p'}(\T^2), \|g\|_{W^{2,p'}}\le 1 \Big\} \\
  & \le \big\| - v^{\ve}\cdot \nabla v^{\ve} - \nabla \cdot (\nabla d^{\ve}\odot \nabla d^{\ve}) \nabla\cdot\sigma^\ve+\nabla\cdot((\nabla\Phi^\ve\otimes\nabla\Phi^\ve)\varepsilon(d^\ve))\big\|_{W_x^{-2,p}}\\
  &\le C .
  \end{align*}
  To estimate $d_{t}^{\ve}$, observe that
  \begin{equation*}
    d_{t}^{\ve}=\od^\ve - v^{\ve}\cdot \nabla
    d^{\ve}+\Omega(v^\ve)d^\ve\in L_{t}^{2}L_{x}^{2}+ L_{t}^{4}L_{x}^{4}\times L_{t}^{2}L_{x}^{2}\subset
    L_{t}^{4/3}L_{x}^{4/3},
  \end{equation*}
  and
  \begin{equation*}
    \begin{split}
      \|d_{t}^{\ve}\|_{L_t^{4/3} L_x^{4/3}}
      &\le \|\od^\ve\|_{L_t^2 L_x^2} + \|v^{\ve}\|_{L_t^4 L_x^4}\|\nabla d^{\ve}\|_{L_t^2 L_x^2}  +\|\nabla v^\ve\|_{L_t^2L_x^2}\le C ,
    \end{split}
  \end{equation*}
  where we have used the energy inequality, the bound $\|\Omega(v^{\ve})d^{\ve}\|_{L^2_tL^2_x}\le \|d^{\ve}\|_{L^\infty(Q_T)}\|\nabla v^{\ve}\|_{L^2_tL^2_x}\le C$, and the embedding
  $L_{t}^{\infty}L_{x}^{2}\cap L_{t}^{2}H_{x}^{1}\subset L_{t}^{4}L_{x}^{4}$, which follows from the Ladyzhenskaya inequality $\|f\|_{L^4(\T^2)}^{2}\lesssim \|f\|_{L^2}\|f\|_{H^1}$. Hence,
  by Aubin--Lions's lemma, applied with
  \begin{equation*}
    \begin{aligned}
      v^{\ve}&:\quad H^{1}_{x}\Subset L^{2}_{x}\hookrightarrow W^{-2,p}_{x,\div}, \qquad
      &&\sup_{\ve}\Big(\|v^{\ve}\|_{L^{2}_{t}H^{1}_{x}}+\|v^{\ve}_{t}\|_{L^{2}_{t}W^{-2,p}_{x,\div}}\Big)<\infty,\\
      d^{\ve}&:\quad H^{1}_{x}\Subset L^{2}_{x}\hookrightarrow L^{4/3}_{x}, \qquad
      &&\sup_{\ve}\Big(\|d^{\ve}\|_{L^{\infty}_{t}H^{1}_{x}}+\|d^{\ve}_{t}\|_{L^{4/3}_{t}L^{4/3}_{x}}\Big)<\infty,
    \end{aligned}
  \end{equation*}
  there exist
  \begin{equation*}
    v \in L_{t}^{\infty}L_{x}^{2}\cap L_{t}^{2}H_{x}^{1}(Q_{T},\mathbb{R}^{2}) ,
    \qquad d \in L_{t}^{\infty}H_{x}^{1}(Q_{T},\R^3),
  \end{equation*}
  such that, after taking a subsequence,
  \begin{equation*}
    (v^{\ve}, d^{\ve}) \to (v,d) \quad \text{in }L^{2}(Q_{T}), \qquad (\nabla v^{\ve}
    , \nabla d^{\ve}) \rightharpoonup (\nabla v, \nabla d) \quad \text{in }L^{2}(Q_{T}).
  \end{equation*}
  This also implies
  \begin{equation*}
    d_{t}^{\ve}+ v^{\ve}\cdot \nabla d^{\ve}\rightharpoonup d_{t}+ v \cdot \nabla d \quad \text{in
    }L^{2}(Q_{T}).
  \end{equation*}
  Furthermore, $(v^{\ve}, d^{\ve})$ is bounded in $C_{w}([0,T], W^{-2,p}(\T^2)
  \times L_{x}^{4/3}(\T^2))$. This is because the uniform bounds on the time derivatives give, for all $0\le s\le t\le T$,
  \begin{equation*}
    \|v^{\ve}(t)-v^{\ve}(s)\|_{W^{-2,p}_{x,\div}}\le\int_{s}^{t}\|v^{\ve}_{t}\|_{W^{-2,p}_{x,\div}}\,\rd\tau\le C|t-s|^{\frac{1}{2}},\qquad
    \|d^{\ve}(t)-d^{\ve}(s)\|_{L^{4/3}_x}\le C|t-s|^{\frac{1}{4}},
  \end{equation*}
  by H\"older's inequality in time, so the two families are in fact equicontinuous with values in $W^{-2,p}_{x,\div}\times L^{4/3}_{x}$. Together with the boundedness
  \begin{equation*}
    (v^{\ve}, d^{\ve}) \in L_{t}^{\infty}L_{x}^{2}(Q_{T}) \times L_{t}^{\infty}H_{x}
    ^{1}(Q_{T}),
  \end{equation*}
  this implies that $(v^{\ve}, d^{\ve})$ is bounded in
  $C_{w}([0,T], L^{2}(\T^2) \times H^{1}(\T^2))$. Thus, after extracting a
  subsequence,
  \begin{equation}\label{eqn:vdconvergence}
    (v^{\ve}(t), d^{\ve}(t)) \to (v(t), d(t)) \quad \text{in }L^{2}(\T^2) \times
    H^{1}(\T^2), \quad \forall\, t \in [0,T].
  \end{equation}
  By lower semicontinuity, for all $t \in (0,T)$,
  \begin{equation*}
    \int_{\T^2\times\{t\}}\big(|v|^{2}+ |\nabla d|^{2}\big)\rd x + \int_0^t\int_{\T^2}\big(|
    \nabla v|^{2}+ |\od(t)|^{2}+|D(v)d|^2\big) \rd x\rd s\le E_{0}.
  \end{equation*}
  Recall from Section \ref{3.2} the set $A\subset[0,T]$ with $|A|=T$ of good time slices, constructed there via Fatou's lemma and the energy inequality, and, for $t\in A$, the finite concentration set $\Sigma_{t}$ with $d^{\ve}(t)\to d(t)$ in $H^{1}_{\rm loc}(\T^{2}\setminus\Sigma_{t})$ after passing to a subsequence.
Next we would like to verify that $d$ is a weak solution by utilizing the geometric structure.
By the convexity of $F$ and the lower semicontinuity, we can show
\begin{equation*}
  \int_{\T^2}F(|d(t)|^2)\dx\le \liminf_{\ve\to 0}\int_{\T^2}F(|d^\ve(t)|^2)\dx\le \lim_{\ve\to 0}C\ve^2=0
\end{equation*}
which leads to the conclusion that $|d(t)|=1$ a.e. in $\T^2$.
Here the last inequality holds with $C=2E_{0}$, since the energy estimate gives $\frac{1}{2\ve^{2}}\int_{\T^2}F(|d^{\ve}(t)|^{2})=\int_{\T^2}\mathcal{F}_{\ve}(d^{\ve}(t))\le E_{0}$; the first inequality follows from Fatou's lemma along a subsequence with $d^{\ve}(t)\to d(t)$ a.e. Since $F\ge0$ vanishes only at $1$, we get $F(|d(t)|^{2})=0$ a.e., that is, $|d(t)|=1$ a.e.
For any $\psi\in C^\infty(\T^2\times \R_+)$ with compact support in $\T^2\times (0,T)$, we can take the cross product
of $d^\ve\psi$ with \eqref{eqn:GinzburgLandauApp}$_{5}$ and integrate over $\T^2\times (0,T)$, then we have
\begin{equation*}
  \begin{split}
    &\int_0^T\int_{\T^2} \gamma_1\od^\ve\wedge d^\ve \psi +\gamma_2 D(v^\ve)d^\ve\wedge d^\ve\psi\rd x\rd t \\
    &=\int_0^T\int_{\T^2} d^\ve\wedge\nabla d^\ve\cdot \nabla \psi +\ve_a(\nabla \Phi^\ve\otimes\nabla \Phi^\ve)d^\ve\wedge d^\ve\psi\rd x\rd t
  \end{split}
\end{equation*}
Here the singular potential term drops out, as $f_{\ve}(d^{\ve})=\frac{1}{\ve^{2}}F'(|d^{\ve}|^{2})\,d^{\ve}$ is parallel to $d^{\ve}$, and the Laplacian term is integrated by parts using $\sum_{i=1}^{2}\partial_{x_i}d^{\ve}\wedge\partial_{x_i}d^{\ve}=0$.

For the term quadratic in $\nabla\Phi^{\ve}$, we use the strong convergence
\begin{equation}\label{phistrong}
  \nabla\Phi^\ve\to\nabla\Phi \quad\text{in } L^p(Q_T)\ \text{for some } p>2,
\end{equation}
which follows from \eqref{phiin}, \eqref{lil4}, and elliptic regularity; see \eqref{phiconv} at the end of this section.

By passing to the limit as $\ve\to 0^+$, noting that each of the remaining terms is the product of a factor converging weakly in $L^{2}(Q_{T})$ ($\od^{\ve}$, $D(v^{\ve})$, $\nabla d^{\ve}$, respectively) with factors converging strongly in $L^{q}(Q_{T})$ for every $q<\infty$ (namely $d^{\ve}$ and its products, since $d^{\ve}\to d$ in $L^{2}(Q_{T})$ and $|d^{\ve}|<2$), while the term quadratic in $\nabla\Phi^{\ve}$ converges in $L^{1}(Q_{T})$ by \eqref{phistrong}, we have
\begin{equation*}
  \begin{split}
    &\int_0^T\int_{\T^2} \gamma_1\od\wedge d \psi +\gamma_2 D(v)d\wedge d\psi\rd x\rd t \\
    &=\int_0^T\int_{\T^2}\nabla d\wedge d\cdot \nabla \psi +\ve_a(\nabla \Phi\otimes\nabla \Phi)d\wedge d\psi\rd x\rd t
  \end{split}
\end{equation*}
This yields 
\begin{equation*}
  [\gamma_1\od+\gamma_2D(v)d-\Delta d-\ve_a(\nabla \Phi\otimes\nabla \Phi)d]\wedge d=0.
\end{equation*}
Here the identity is understood in the sense of distributions, with $\Delta d\wedge d:=\nabla\cdot(\nabla d\wedge d)$, which is meaningful since $\nabla d\wedge d\in L^{\infty}_{t}L^{2}_{x}(Q_{T})$. Since $|d|=1$ a.e., for a.e. $(x,t)$ a vector $w\in\R^{3}$ satisfies $w\wedge d=0$ if and only if $w=(w\cdot d)\,d$.
Hence there is a multiplier $\lambda:\T^2\times(0,T)\to \R$ such that
\begin{equation*}
  \gamma_1\od+\gamma_2D(v)d-\Delta d-\ve_a(\nabla \Phi\otimes\nabla \Phi)d=\lambda d.
\end{equation*}
By the fact that $|d|=1$, we can multiply the equation by $d\psi$ and integrate over $\T^2$ to yield
\begin{equation*}
  \begin{split}
    &\int_{\T^2}\big(\gamma_2(d\cdot D(v)d)+|\nabla d|^2-\ve_a(\nabla \Phi\otimes\nabla \Phi)d\cdot d\big)\psi\rd x=\int_{\T^2}\lambda \psi\rd x.
  \end{split}
\end{equation*}
Then we have $\lambda=\gamma_2(d\cdot D(v)d)-|\nabla d|^2+\ve_a(\nabla \Phi\otimes\nabla \Phi)d\cdot d$.
In the computation above we used the identities, valid whenever $|d|=1$ a.e.: $\od\cdot d=\frac{1}{2}(\partial_{t}+v\cdot\nabla)|d|^{2}-\Omega(v)d\cdot d=0$, since $\Omega(v)$ is antisymmetric, and $\Delta d\cdot d=\frac{1}{2}\Delta|d|^{2}-|\nabla d|^{2}=-|\nabla d|^{2}$, both in the sense of distributions.

Next we check $v$ is a weak solution to \eqref{eqn:MainPDE}$_{3}$.
For any $z\in C^\infty((0, T)\times\T^2;\R^2)$ with $\nabla\cdot z=0$, we have
\begin{equation*}
  \begin{split}
    \int_0^T\int_{\T^2}&v^\ve\cdot\pa_t z+(v^\ve\otimes v^\ve):\nabla z\rd x\rd t \\
    &=\int_0^T\int_{\T^2}[-\nabla d^\ve\odot \nabla d^\ve+\sigma^\ve+(\nabla \Phi^\ve\otimes\nabla \Phi^\ve)\varepsilon(d^\ve)]:\nabla z\rd x\rd t
  \end{split}
\end{equation*}
For a.e. $t\in A$: $v^{\ve}(t)\otimes v^{\ve}(t)\to v(t)\otimes v(t)$ in $L^{1}(\T^{2})$; $\sigma^{\ve}(t)\rightharpoonup\sigma(t)$ in $L^{4/3}(\T^{2})$ (weakly-$L^{2}$ convergent factors $\od^{\ve},D(v^{\ve})$ times strongly convergent factors of $d^{\ve}$); $\big((\nabla\Phi^{\ve}\otimes\nabla\Phi^{\ve})\varepsilon(d^{\ve})\big)(t)\to\big((\nabla\Phi\otimes\nabla\Phi)\varepsilon(d)\big)(t)$ in $L^{2}(\T^{2})$ by \eqref{phiconv}; and the Ericksen stress converges by \eqref{eqn:EricksenConv}, where subtracting the null term $-\frac{1}{2}|\nabla d^{\ve}|^{2}{\rm Id}$ costs nothing since ${\rm Id}:\nabla z=\nabla\cdot z=0$.

Notice that for the good time slices $A\subset (0,T)$ with $|A|=T$, we have
$\nabla d^\ve(t)\odot\nabla d^\ve(t)-\frac12 |\nabla d^\ve(t)|^2\rightharpoonup \nabla d(t)\odot \nabla d(t)-\frac12|\nabla d(t)|^2$ for $t\in A$, 
hence,  passing the limit as $\ve\to 0^+$ we have
\begin{equation*}
  \begin{split}
    \int_{\T^2\times\{t\}}&v\cdot\pa_t z+(v\otimes v):\nabla z\rd x \\
    &=\int_{\T^2\times\{t\}}[-\nabla d\odot \nabla d+\sigma+(\nabla \Phi\otimes\nabla \Phi)\varepsilon(d)]:\nabla z\rd x.
  \end{split}
\end{equation*}
Integrating it over $A$ and using the fact that $|A|=T$, we have shown that $v$ satisfies the weak form of \eqref{eqn:MainPDE}$_3$.

Lastly, we note that due to \eqref{alllp} and \eqref{phiin}, we have that $\ck$ and $\nabla\Phi^\ve$ are bounded uniformly in $L^\infty(Q_T)$. Furthermore, from \eqref{nac}, we know that $\ck$ is bounded uniformly in $L^\infty_tH^1_x.$ Then, from the Nernst-Planck equations $\eqref{eqn:GinzburgLandauApp1}_1$, we have
\begin{align*}
    \|\pa_t \ck\|_{L^\frac{4}{3}_tH^{-1}_x}&=\|-v^{\epsilon}\cdot \nabla \ck                +\cD_{k}\Delta\ck+\nabla\cdot(\ck \cD_{k}\nabla \Phi_k) \|_{L^\frac{4}{3}_tH^{-1}_x}\\
    &\lesssim \|v^\ve\cdot\nabla \ck\|_{L^\frac{4}{3}_tL^\frac{4}{3}_x}+\|\ck\|_{L^\frac{4}{3}_tH^1_x}+\|\ck \nabla\Phi_k\|_{L^\frac{4}{3}_tL^2_x}\\
    &\lesssim\|v^\ve\|_{L^4_tL_x^4}\|\nabla \ck\|_{L^2_tL^2_x}+ \|\ck\|_{L^2_tH^1_x}+\|\ck\|_{L^\infty_tL^\infty_x}\|\nabla\Phi^\ve\|_{L^\infty_t L^\infty_x}.
\end{align*}
Then using again the fact that $v^\ve$ is bounded uniformly in $L^\infty_tL^2_x\cap L^2_tH^1_x\subset L^4_tL^4_x$, we conclude that $\pa_t \ck$ is bounded uniformly in $L^\frac{4}{3}_t H^{-1}_x$. So by Aubin-Lions's lemma (or rather its refinement, due to Simon \cite{simon}), there exists
$$c_k \in L^\infty_tL^\infty_x\cap L^2_t H^1_x$$
such that, after passing to a subsequence,
\begin{align}\label{lil4}
\ck \to c_k \quad \text{in } L^\infty_t L^4_x.
\end{align}
Now with $c_k$ and $d$ defined, $\Phi$ is defined via $\eqref{eqn:MainPDE}_2$. Also, using elliptic regularity for $\Phi^\ve$ together with \eqref{lil4} and using the fact that $\nabla \ck$ is bounded, in particular, in $L^2(Q_T)$, we have (after passing to another subsequence if necessary)
\begin{align}\label{phiconv}
\nabla \Phi^\ve \to  \nabla\Phi \quad \text{in }L^{\infty}(Q_{T}), \qquad \nabla \ck\rightharpoonup \nabla c_k \quad \text{in }L^{2}(Q_{T}).
\end{align}
Then, since we know from above that $(v^\ve,d^\ve)\to (v,d)$ in $L^2(Q_T)$, we have 
\begin{align*}
\ck v^\ve &\to  c_k v \quad \text{in }L^{1}(Q_{T}),\\
\varepsilon(d^\ve)\nabla\Phi^\ve &\to \varepsilon(d)\nabla \Phi \quad \text{in }L^{1}(Q_{T}),\\
\ck\mathcal{D}_k\nabla\Phi^\ve &\to c_k\mathcal{D}_k \nabla\Phi \quad \text{in } L^2(Q_T).
\end{align*}
Now, taking $\phi_k,\psi\in C^\infty_c((0,T)\times \mathbb{T}^2)$ and testing $\eqref{eqn:GinzburgLandauApp1}_1$ and $\eqref{eqn:GinzburgLandauApp1}_2$ against these test functions, respectively, we obtain
\begin{equation}
\begin{aligned}\label{ints}
\int_0^T\int_{\mathbb{T}^2} \ck \pa_t\phi_k+\ck v^\ve\cdot\nabla\phi_k\rd x\rd t &= \int_0^T\int_{\mathbb{T}^2} \ck\mathcal{D}_k\nabla\mu_k^\ve\cdot\nabla\phi_k \rd x\rd t\\
\int_0^T\int_{\mathbb{T}^2}(\varepsilon(d^\ve)\nabla\Phi^\ve)\cdot\nabla\psi\rd x\rd t &= -\int_0^T\int_{\mathbb{T}^2}\sum_{k=1}^N z_kc_k\psi \rd x\rd t.
\end{aligned}
\end{equation}
Observing that $\nabla\mu_k^\ve=\frac{\nabla \ck}{\ck}+z_k\nabla\Phi^\ve$ and $\nabla\mu_k=\frac{\nabla c_k}{c_k}+z_k\nabla\Phi$, and using the convergence results stated above, we conclude, by passing to the limit $\ve\to 0^+$ in \eqref{ints}, that $c_k, \Phi, v, d$ satisfy $\eqref{eqn:MainPDE}_1$ and $\eqref{eqn:MainPDE}_2$ in the weak sense.

It remains to verify that the limit $(c_{1},\dots,c_{N},\Phi,v,d)$ meets the requirements of Definition \ref{def:weak}. The integral identities have been established above. For the regularity: $v\in L^{\infty}_{t}L^{2}_{{\rm div},x}\cap L^{2}_{t}H^{1}_{{\rm div},x}$ and $d\in L^{\infty}_{t}H^{1}_{x}$ with $|d|=1$ a.e.\ were obtained at the beginning of this section; $c_{k}\in L^{\infty}_{t}L^{\infty}_{x}\cap L^{2}_{t}H^{1}_{x}$ follows from \eqref{alllp}, \eqref{ch1}, and \eqref{lil4} by lower semicontinuity; and $\Phi\in L^{\infty}_{t}H^{1}_{x}\cap L^{\infty}_{t}L^{\infty}_{x}$ follows from \eqref{phiin} and elliptic regularity (normalizing $\int_{\T^2}\Phi\,\dx=0$). The nonnegativity $c_{k}\ge0$ is inherited from the positivity of $\ck$ (Section 2) via \eqref{lil4}. Finally, the initial conditions are attained: $(v^{\ve},d^{\ve})$ is bounded in $C_{w}([0,T],L^{2}\times H^{1})$ with common initial data, so $v(0)=v_{0}$ and $d(0)=d_{0}$, while $c_{k}(0)=c_{k,0}$ follows from the uniform bound on $\pa_{t}\ck$ in $L^{4/3}_{t}H^{-1}_{x}$, which renders $\ck$ equicontinuous with values in $H^{-1}_{x}$. This completes the proof of Theorem \ref{thm:MainTheorem}. \qed

\providecommand{\bysame}{\leavevmode\hbox to3em{\hrulefill}\thinspace}
\providecommand{\href}[2]{#2}

\end{document}